\documentclass[preprint,3p,authoryear]{elsarticle}
\journal{Journal of Inequalities and Applications}
\usepackage{amsthm,amssymb,amsmath}
\usepackage{enumerate}
\usepackage{algorithm}
\usepackage{algorithmic}
\usepackage{multirow}
\usepackage{xcolor}
\usepackage{color}
\DeclareMathOperator*{\argmin}{argmin}

\usepackage{url}
\usepackage{cases}
\usepackage{hyperref}

\newtheorem{theorem}{Theorem}[section]
\newtheorem{definition}{Definition}[section]
\newtheorem{proposition}{Proposition}[section]
\newtheorem{corollary}{Corollary}[section]
\newtheorem{lemma}{Lemma}[section]

\newtheorem{remark}{Remark}

\usepackage{geometry}
\begin{document}

\begin{frontmatter}
\title{\large Compound Poisson Point Processes, Concentration and Oracle Inequalities}
\small{
\address[label1]{School of Mathematical Sciences and Center for Statistical Science, Peking University, Beijing, 100871, China}
\author[label1]{Huiming Zhang}
\ead{zhanghuiming@pku.edu.cn}
\address[label2]{Mathematics Department, Rutgers University, Piscataway, NJ 08854, USA}
\author[label2]{Xiaoxu Wu}
\ead{xw292@rutgers.edu}

\begin{abstract}
This note aims at presenting several new theoretical results for the
compound Poisson point process, which follows the work of Zhang \emph{et al.} [Insurance~Math.~Econom.~59(2014), 325-336]. The first part
provides a new characterization for a discrete compound Poisson point
process (proposed by {Acz{\'e}l} [Acta~Math.~Hungar.~3(3)(1952), 219-224]),
it extends the characterization of the Poisson point
process given by Copeland and Regan [Ann.~Math.~(1936): 357-362]. Next,
we derive some concentration inequalities for discrete compound Poisson
point process (negative binomial random variable with unknown
dispersion is a significant example). These concentration inequalities
are potentially useful in count data regressions. We give an application
in the weighted Lasso penalized negative binomial regression whose KKT
conditions of penalized likelihood hold with high probability and then
we derive non-asymptotic oracle inequalities for a weighted Lasso
estimator.

\end{abstract}

}
\end{frontmatter}

\section{Introduction}

Powerful and useful concentration inequalities of empirical processes
have been derived one after another for several decades. Various
types of concentration inequalities are mainly based on moments
condition (such as sub-Gaussian, sub-exponential and sub-Gamma) and bounded difference condition; see
\cite{Boucheron13}. Its central task is to evaluate the
fluctuation of empirical processes from some value (for example, the
mean and median) in probability. Attention has greatly expanded in
various area of research such as high-dimensional statistics, models, etc. For the Poisson or negative binomial count data regression (\cite{Hilbe11}), it should be noted that the Poisson
distribution (the limit distribution of negative binomial is Poisson)
is not sub-Gaussian, but locally sub-Gaussian distribution
(\cite{Chareka06}) or sub-exponential distribution
(see Sect.~2.1.3 of \cite{Wainwright19}). However, few relevant results are recognized about the concentration inequalities for a weighted sum of negative binomial (NB) random variables and its statistical applications. With applications to the segmentation of RNA-Seq data, a documental example is that \cite{Cleynen14} obtained a concentration inequality of the sum of independent centered negative binomial random variables and its derivation hinges on the assumption that the dispersion parameter is known. Thus, the negative binomial random variable belongs to the exponential family. But, if the dispersion parameter in negative binomial random variables $X$ is unknown in real world problems, it
does not belong to the exponential family with density
\begin{equation}
\label{eq:exponential} f_{X}(x\mid \theta ) = h(x) \exp \bigl(\eta (
\theta )T(x) -A(\theta ) \bigr),
\end{equation}
where $T(x)$, $h(x)$, $\eta (\theta )$, and $A(\theta )$ are some given
functions.

Whereas, it is well-known that Poisson and negative binomial distributions belong to the family of infinitely divisible distributions. \cite{Houdre02a} studies dimension-free concentration inequalities for non-Gaussian infinitely divisible random vectors with finite exponential moments. In particular, the geometric case, a special case of negative binomial, has been obtained in \cite{Houdre02a}. Via
the entropy method, \cite{Kontoyiannis06} gives a simple derivation
of the concentration inequalities for a Lipschitz function of discrete compound Poisson random
variable. Nonetheless, when deriving oracle inequalities of Lasso or
Elastic-net estimates (see \cite{Ivanoff16}, Zhang and Jia
\cite{Zhang17}), it is strenuous to use the results of
\cite{Kontoyiannis06}, \cite{Houdre02a} to get an explicit
expression of tuning parameter under the KKT conditions of the penalized
negative binomial likelihood (Poisson regression is a particular case).
Moreover, when the negative binomial responses are treated
as sub-exponential distributions (distributions with the Cramer type
condition), the upper bound of the $\ell _{1}$-estimation error oracle
inequality is determined by the tuning parameter with the rate
$O( {\frac{{\log p}}{n}} )$, which does not show rate-optimality in the
minimax sense, although it is sharper than $O(\sqrt{
\frac{{\log p}}{n}} )$. Consider linear regressions $Y = X\beta ^{*} +
\varepsilon $ with noise $\operatorname{Var} \varepsilon = {\sigma
^{2}}{I_{n}}$ and $\beta ^{*}$ being the $p$-dimensional true parameter.
Under the $\ell _{0}$-constraints on $\beta ^{*}$, \cite{Raskutti2011} shows that the lower bound on the minimax prediction error for all estimators $\hat{\beta }$ is
\begin{equation*}
\inf _{\hat{\beta}} \sup _{\|\beta^*\|_{0} \leq s_{0}} \frac{1}{n}\|X \hat{\beta}-X \beta^*\|_{2}^{2} \ge O(\frac{s_{0}\log \left(p / s_{0}\right)}{n}),\quad (X \text{ is an } n \times p \text{ design matrix})
\end{equation*}
with probability at least $1/2$. Thus, the optimal minimax rate for
$\| \hat{\beta }-\beta ^{*}\|_{1}$ should be proportional to the root of
the rate $O( {\frac{{\log p}}{n}} )$. Similar minimax lower bound for
Poisson regression is provided in \cite{Jiang2015};
see Chap.~4 of \cite{Rigollet19} for more minimax theory.

This paper is prompted by the motif of high-dimensional count regression
that derives useful NB concentration inequalities which can be obtained
under compound Poisson frameworks. The obtained concentration
inequalities are beneficial to optimal high-dimensional inferences. Our
concentration inequalities are about the sum of independent centered
weighted NB random variables, which is more general than the
un-weighted case in \cite{Cleynen14}, which supposes that the
over-dispersion parameter of NB random variables is known. However, in
our case, the over-dispersion parameter can be unknown. As a by-product,
we proposed a characterization for discrete compound Poisson point
processes following the work in \cite{Zhang14} and \cite{Zhang16}. The paper will end up with some applications of high-dimensional NB regressions. \cite{Stadler10} studies the oracle inequalities for $\ell_{1}$-penalization finite mixture of Gaussian regressions model but they didn't concern the mixture of count data regression (such as mixture Poisson, see \cite{Yang19} for example). Here the NB regression can be approximately viewed as finite mixture of Poisson regression since NB distribution is equivalent to a continuous mixture of Poisson distributions where the mixing distribution of the Poisson rate is gamma distributed. The paper will end up with some applications of high-dimensional NB regression in terms of oracle inequalities.

The paper is organized as follows. Section~\ref{sec2} is an introduction to discrete compound Poisson point processes (DCPP). We give a theorem for characterizing DCPP, which is similar to the result that initial condition, stationary and independent increments, condition of jumps, together quadruply characterize a compound Poisson process, see \cite{Wang1993} and the
references therein. In Sect.~\ref{sec3}, we derive the concentration inequality for discrete compound
Poisson point process with infinite-dimensional parameters. As an application, for the optimization
problem of weighted Lasso penalized negative binomial regression, we show that the true parameter version of
KKT conditions hold with high probability, and the optimal data-driven weights in weighted Lasso penality
are determined. We also show the oracle inequalities for weighted Lasso estimates in NB regressions,
and the convergence rate attain the minimax rate derived in references.

\section{Introductions to discrete compound Poisson point process}\label{sec2}
In this section, we present the preliminary knowledge for discrete compound Poisson point process. A new characterization of the point process with five assumptions is obtained.

\subsection{Definition and preliminaries}
To begin with, we need to be familiar with the definition of the
discrete compound Poisson distribution and its relation to the weighted sum
of Poisson random variables. For more details, we refer the reader to
Sect.~9.3 of \cite{johnson05}, \cite{Zhang14}, \cite{Zhang16} and the references therein.

\begin{definition}
\label{def-dcp}
We say that $Y$ is discrete compound Poisson (DCP) distributed if the
characteristic function of $Y$ is
\begin{equation}
\label{eq:cf} {\varphi _{Y}}(t) = \mathrm{{E}}
{e^{\mathrm{{i}}tY}} = \exp \Biggl\{ \sum_{i = 1}^{\infty }{{
\alpha _{i}}\lambda \bigl({e^{\mathrm{{i}}t
\theta }} - 1 \bigr)} \Biggr\}
\quad ( \theta \in \mathbb{R}),
\end{equation}
where $({\alpha _{1}}\lambda ,{\alpha _{2}}\lambda , \ldots )$ are
infinite-dimensional parameters satisfying $\sum_{i = 1}^{\infty }
{\alpha _{i}} = 1$, ${\alpha _{i}} \ge 0$, \mbox{$\lambda > 0$}. We denote it by
$Y \sim \operatorname{DCP}({\alpha _{1}}\lambda ,{\alpha _{2}}\lambda ,
\ldots )$.
\end{definition}

If ${\alpha _{i}}=0$ for all $i > r$ and ${\alpha _{r}} \ne 0$, we say
that $Y$ is DCP distributed of order $r$. If $r =  +
\infty $, then the DCP distribution has infinite-dimensional parameters
and we say it is DCP distributed of order $ + \infty $. When
$r=1$, it is the well-known Poisson distribution for modeling
equal-dispersed count data. When $r=2$, we call it a Hermite
distribution, which can be applied to model the over-dispersed and
multi-modal count data; see \cite{Giles10}.

Equation \eqref{eq:cf} is the canonical representation of characteristic
function for non-negative valued infinitely divisible random variable,
i.e. a special case of the L\'{e}vy--Khinchine formula (see Chap.~2
of \cite{Sato13} and Sect.~1.6 of \cite{Petrov1995}),
\begin{equation*}
{\mathrm{{E}}} {e^{\mathrm{{i}}t Y}} = \exp \biggl( {a\mathrm{{i}}t -
\frac{1}{2}{\sigma ^{2}} {t^{2}} +
\int _{\mathbb{R}\backslash \{ 0\} } { \bigl({e^{\mathrm{{i}}tx}} - 1 -\mathrm{{i}}t x{
{\rm{I}}_{| x | < 1}}(x) \bigr) \nu (dx)} } \biggr),
\end{equation*}
where $a \in \mathbb{R}$, $\sigma \ge 0$ and $\mathrm{I}_{\{\cdot \}}(x)$
is indicator function, the function $\nu (\cdot )$ is called the
L\'{e}vy measure with restriction: $\int _{\mathbb{R}\backslash \{ 0
\} } {\min ({x^{2}},1)\nu(dx)} < \infty $.

Let ${\delta _{x}}(t) := 1_{\{ x\}}(t)$ and $\mathbb{N}^{+}$ be the
positive integer set. If we set $\nu (t) = \sum_{x \in \mathbb{N}^{+}} {\lambda {\delta _{x}}(t){\alpha _{x}}}
$, the $\nu (\cdot )$ is the L\'{e}vy measure in the
L\'{e}vy--Khinchine formula. It is easy to see that $Y$ has a weighted
Poisson decomposition: $Y = \sum_{i = 1}^{\infty }{i{N_{i}}} $,
where the ${{N_{i}}}$ are independent Poisson distributed with mean
$\lambda {\alpha _{i}}$. This decomposition is also called a
L\'{e}vy--Ito decomposition; see Chap.~4 of Sato \cite{Sato13}. If
$\mathrm{{E}}{e^{ - \theta Y}}<\infty $ for $\theta $ in a neighbourhood
of zero, the moment generating function (m.g.f.) of DCP is\vspace*{-2pt}
\begin{equation*}
{M_Y}(\theta ):={\rm{E}}{e^{ - \theta Y}} = \exp \{ \sum\limits_{k = 1}^\infty  {{\alpha _k}\lambda ({e^{ - k\theta }} - 1)} \}= \exp \{ \int_0^\infty  {({e^{-\theta x}} - 1)} \nu ({\rm{dx}})\} ,~~(|\theta|< R)
\end{equation*}
for some $R>0$.

In order to define the discrete compound Poisson point process, we shall
repeatedly use the concept of Poisson random measure. Let $A$ be any
measurable set on a measurable space $(E, \mathcal{A})$, a good example could be  $E=\mathbb{R}^{d}$. In the following sections, we let $E=\mathbb{R}
^{d}$ and denote $N(A,\omega )$ as the number of random points in set
$A$. We introduce the Poisson point processes below, and sometimes it
is called the Poisson random measure (\cite{Sato13},
Kingman \cite{Kingman93}).\vspace*{-2pt}
\begin{definition}
\label{def-pps}
Let $(E, \mathcal{A}, \mu )$ be a measurable space with $\sigma $-finite
measure $\mu $, and $\mathbb{N}$ be the non-negative integer set. The
Poisson random measure with intensity $\mu $ is a family of random
variables $\{N(A,\omega )\}_{{A\in {\mathcal{A}}}}$ (defined on some
probability space $(\varOmega , \mathcal{F}, P)$) as the product map\vspace*{-2pt}
\begin{equation*}
N : \mathcal{A} \times \varOmega \rightarrow \mathbb{N}
\end{equation*}
satisfying
\begin{enumerate}[3.]
\item[1.]
$\forall A\in \mathcal{A}$, the $N(A, \cdot )$ is Poisson random
variable on $(\varOmega , \mathcal{F}, P)$ with mean $\mu (A)$;
\item[2.]
$\forall \omega \in \varOmega $, the $N( \cdot ,\omega )$ is a counting
measure on $(E, \mathcal{A})$;
\item[3.]
If sets $A_{1},A_{2},\ldots ,A_{n}\in \mathcal{A}$ are disjoint, then
$N({A_{1}}, \cdot ), N({A_{2}}, \cdot ), \ldots , N({A_{n}}, \cdot )$ are
mutually independent.
\end{enumerate}
\end{definition}

In order to define the discrete compound Poisson point process, we shall repeatedly use the concept of Poisson random measure. Let $A$ be any measurable set on a measurable space $(E, \mathcal{A})$. A good example is $E=\mathbb{R}^d$. In the following section, we let $E=\mathbb{R}^d$ and denote $N(A,\omega)$ as the number of random points in set $A$. We introduce the Poisson point processes below, and sometimes it is called the Poisson random measure~(\cite{Sato13}, \cite{Kingman93}).
\begin{definition} \label{def-pps}
Let $(E, \mathcal A, \mu)$ be a measurable space with $\sigma$-finite measure $\mu$, and $\mathbb{N}$ be the non-negative integer set. The Poisson random measure with intensity $\mu$ is a family of random variables $\{N(A,\omega )\}_{{A\in {\mathcal  {A}}}}$ (defined on some probability space $(\Omega, \mathcal{F}, P)$) as the product map
$$
N :  \mathcal{A} \times \Omega  \rightarrow \mathbb{N}
$$
satisfying
\begin{enumerate}
\item
$\forall A\in\mathcal{A}$, the $N(A, \cdot )$ is Poisson random variable on $(\Omega, \mathcal{F}, P)$ with mean $\mu(A)$;
\item
$\forall \omega\in\Omega$, the $N( \cdot ,\omega )$ is a counting measure on $(E, \mathcal A)$;
\item
If sets $A_1,A_2,\ldots,A_n\in\mathcal{A}$ are disjoint, then $N({A_1}, \cdot ),N({A_2}, \cdot ), \cdots ,N({A_n}, \cdot )$ are mutually independent.
\end{enumerate}
\end{definition}

Let $N(\mathcal{A}):=N(\mathcal{A},\omega )$. Based on the Poisson
random measure, we define the discrete compound Poisson point processes
(DCPP) $\{CP( A)\}_{{A\in {\mathcal{A}}}}$ by the L\'{e}vy--Ito
decomposition\vspace*{-2pt}
\begin{equation*}
CP(A) := \sum_{k = 1}^{\infty }{k{N_{k}}(A)},
\end{equation*}
where the ${N_{k}}(A)$ are independent Poisson point process with
mean measure $\mu _{k} (A): = {\alpha _{k}}\*\int _{A} {\lambda (x)} \,dx$.

The m.g.f. of $\{CP(A)\}_{{A\in {\mathcal{A}}}}$ is ${M_{CP(A)}}(
\theta ) = \exp \{ \sum_{k = 1}^{\infty }{{\alpha _{k}}\int _{A}
{\lambda (x)} \,dx({e^{ - k\theta }} - 1)} \}$. We can see that
$\lambda (\cdot )$ is the intensity function of the \emph{generating
Poisson point processes} $\{{N_{k}}(A)\}_{{A\in {\mathcal{A}}}}$ for
each $k$. \cite{Aczel52} derives the p.m.f. and he calls it the
inhomogeneous composed Poisson distribution as $d=1$. If the intensity function
${\lambda (x)}$ is the constant (i.e. the mean measure is a multiple of Lebesgue measure), the $N$ is said to be a homogeneous Poisson point process.

Define the probability $P_{k}(A)$ by $P(N(A)=k)$. If $N(A)$ follows from
a DCP point process with intensity measure $\lambda (A)$, then the
probability of having $k$ ($k \geq 0$) points in the set $A$ is given by \vspace*{-2pt}
\begin{equation}
\label{eq:xu1} P_{k}(A)=\sum_{R(s,k)=k}
\frac{\alpha _{1}^{s_{1}}\cdots \alpha
_{k}^{s_{k}}}{s_{1}!\cdots s_{k}!} \bigl[\lambda (A) \bigr]^{s_{1}+\cdots +s_{k}}e
^{-\lambda (A)},
\end{equation}
where $R(s,k):=\sum_{t=1}^{k} ts_{t}$; see \cite{Aczel52} for the proof.\vadjust{\goodbreak}

\subsection{A new characterisation of DCP point process}\label{sec2.2}

Based on preliminaries, we give a new characterization of DCP point
process, which is an extension of \cite{Copeland36}. A similar
characterization for DCP distribution and process (not for point process
in terms of random measure) is derived by \cite{Wang1993}. For
more characterizations of the point process, see the monograph
\cite{Last17} and the references therein.

\begin{theorem}[Characterization for DCP point process]
\label{thm:ccdp}%
Consider the following
assumptions:
\begin{enumerate}[5.]
\item[1.] $P_{k}(A)>0$ if $0<m(A)<\infty $, where $m(\cdot )$ is the
Lebesgue measure;

\item[2.] $\sum_{k=0}^{\infty }P_{k}(A)=1$;

\item[3.] $P_{k}(A_{1}\cup A_{2})=\sum_{i=0}^{k} P_{k-i}(A
_{1}) P_{i}(A_{2})$ if $A_{1}\cap A_{2}=\emptyset $;

\item[4.] Let $S_{k}(A)=\sum_{i=k}^{\infty }P_{i}(A)$, then
$\lim_{m(A)\to 0} S_{1}(A)=0$;

\item[5.] $\lim_{m(A)\to 0}\frac{P_{k}(A)}{S_{1}(A)}=\alpha
_{k}$, where $\sum_{k=1}^{\infty }\alpha _{k}=1$.
\end{enumerate}
If $P_{k}(A)$ satisfies the assumptions 1--5, then there exists
a measure $\nu (A)$ (countable additivity, absolutely continuous) such
that $P_{k}(A)$ is represented by the equation (\ref{eq:xu1}) for all
$k$ and $A$.
\end{theorem}

For $A$ being the closed interval in $\mathbb{R}$, this setting turns
to the famous characterization of the Poisson process; see p.~447 of
\cite{Feller68} (The postulates for the Poisson process). The proof of
Theorem~\ref{thm:ccdp} is given in the \hyperref[sec6]{Appendix}. The proof consists of
showing countably additive and absolutely continuous of the intensity
measure and dealing with the matrix differential equation.

As a random measure, the DCPP is a mathematical extension from discrete
compound Poisson distribution and discrete compound Poisson process
index by real line. It has been studied by \cite{Baraud2009} that
they use histogram type estimator to estimate the intensity of the
Poisson random measure. A recent application, \cite{Das2018}
proposes a nonparametric Bayesian approach to RNA-seq count data by DCP
process. For a set of regions $\{A \in {\mathcal{A}}\}$ of interest on
the genome, such as genes, exons, or junctions, the number of reads is
counted, and the count in the region $A_{i}$ is viewed as a DCP process
for gene expression. The $P_{k}(A)$ satisfies assumptions 1--5 in
Theorem~\ref{thm:ccdp} by the real-world setting (the process of generating the
expression level for each gene in the experiment).

\section{Concentration inequalities for discrete compound Poisson point process}\label{sec3}
This section is about the construction of concentration inequalities for
DCPP. It has applications in high-dimensional $\ell _{1}$-regularized
negative binomial regression; see Sect.~\ref{sec4}.

As an appetizing example, if the DCPP has finite parameters, the
existing result such as Lemma~3 of \cite{Baraud2009} can be used
to obtain the concentration inequality for the sum of DCP random variables.

\begin{lemma}[\cite{Baraud2009}]
\label{lm:Baraud}
Let $Y_{1},\ldots ,Y_{n}$ be $n$ independent
centered random variables, and $\tau $, $\kappa $, $\{ {\eta _{i}}\}
_{i = 1}^{n}$ be some positive constants.
\begin{enumerate}[(b)]
\item[(a)] If $\log (\mathrm{{E}}e^{zY_{i}})\leq \kappa \frac{z^{2}\eta
_{i}}{2(1-z\tau )} $ for all $z\in [0,1/\tau [$, and $1\leq i\leq n$,
then
\begin{equation*}
P \Biggl[\sum_{i=1}^{n}
Y_{i} \geq \Biggl(2\kappa x\sum_{i=1}^{n}
\eta _{i} \Biggr)^{1/2} +\tau x \Biggr]\leq
e^{-x}\quad \textit{for all } x>0.
\end{equation*}
\item[(b)] If for $1\leq i\leq n$ and all $z>0$ $\log (\mathrm{{E}}e^{-zY
_{i}})\leq \kappa z^{2}\eta _{i}/2$, then
\begin{equation*}
P \Biggl[\sum_{i=1}^{n}
Y_{i} \leq - \Biggl(2\kappa x\sum_{i=1}^{n}
\eta _{i} \Biggr)^{1/2} \Biggr]\leq e^{-x}
\quad \textit{for all } x>0.
\end{equation*}
\end{enumerate}
\end{lemma}

Actually, the moment conditions in Lemma~\ref{lm:Baraud} is a direct
application of the existing result based on Cramer-type conditions
(Lemma~2.2 in \cite{Petrov1995}, or sub-exponential condition
\eqref{eq:sube} discussed below) for the convolutions' Poisson
distribution with different rates. The centered random variable
$Y_{i}$ with the moment condition $\log (\mathrm{{E}}e^{zY})\leq
\kappa \frac{z^{2}\eta _{i}}{2(1-z\tau )} $ in Lemma~\ref{lm:Baraud}(a),
is called the \emph{sub-Gamma} random variables (see Sect.~2.4 in
\cite{Boucheron13}). A random variable $X$ with zero mean is
\emph{sub-exponential} (see Sect.~2.1.3 in
\cite{Wainwright19}) if there are non-negative parameters $(v, \alpha )$ such that

\begin{equation}
\label{eq:sube} {\mathrm{{E}}}e^{\lambda X} \leq e^{\frac{v^{2}\lambda ^{2}}{2}} \quad
\text{for all } |\lambda| <\frac{1}{\alpha }.
\end{equation}
Note that the sub-exponential implies the sub-gamma condition by the
following inequality:
\begin{center}
$\log (\mathrm{{E}}{e^{ - z{Y_{i}}}}) \le \frac{
{\kappa {z^{2}}{\eta _{i}}}}{2} \le \frac{{\kappa {z^{2}}{\eta _{i}}}}{
{2(1 - z\tau )}}$ for all $z\in [0,1/\tau [$.
\end{center}
\begin{proposition}
\label{thm:cdp}
Let ${Y_{i}} \sim \operatorname{DCP}({\alpha _{1}}(i)\lambda (i), \ldots
,{\alpha _{r}}(i)\lambda (i))$ for $i = 1,2, \ldots ,n$, and
${\sigma _{i}} := \operatorname{Var}{Y_{i}} = \lambda (i)\sum_{k = 1}
^{r} {{k^{2}}{\alpha _{k}}(i)} $, then for all $x > 0$
\begin{gather*}
P \Biggl[ {\sum_{i = 1}^{n}
{({Y_{i}}} - \mathrm{{E}} {Y_{i}}) \ge \sqrt{2x\sum
_{i = 1}^{n} {\sigma _{i}^{2}}
} + rx} \Biggr] \le {e^{ - x}},~P \Biggl[ {\sum
_{i = 1}^{n} {({Y_{i}}} - \mathrm{
{E}} {Y_{i}}) \le - \sqrt{2x\sum_{i = 1}^{n}
{\sigma _{i}^{2}} } } \Biggr] \le {e^{ - x}}.
\end{gather*}
Moreover, we have
\begin{equation} \label{eq:ci1}
P\left[ {\left| {\sum\limits_{i = 1}^n {({Y_i} - {\rm{E}}{Y_i})} } \right| \ge \sqrt {2x\sum\limits_{i = 1}^n {\sigma _i^2} }  + rx} \right] \le 2{e^{ - x}}.
\end{equation}
\end{proposition}

\begin{proof}
First, in order to apply the above lemma, we need to evaluate the
log-moment-generating function of centered DCP random variables. Let
${\mu _{i}} =: \mathrm{{E}}{Y_{i}} = \lambda (i)\sum_{k = 1}
^{r} {k{\alpha _{k}}(i)}$, we have
\begin{align*}
\log {\mathrm{{E}}} {e^{z({Y_{t}} - {\mu _{i}})}} &= - z{\mu _{i}} + \log
\mathrm{E} {e^{z{Y_{i}}}} = - z{\mu _{i}} + \log
{e^{\lambda (i)\sum _{k = 1}^{r} {{\alpha _{k}}(i)} ({e^{kz}}
- 1)}}
\\
&= \lambda (i)\sum_{k = 1}^{r}
{{\alpha _{k}}(i) \bigl({e^{kz}} - kz - 1 \bigr)} .
\end{align*}
Using the inequality $e^{z}-z-1 \leq \frac{ z^{2}}{ 2(1-z)}$ for
$1>z>0$, one derives
\begin{equation*}
\log {\mathrm{{E}}} {e^{z({Y_{t}} - {\mu _{i}})}} \le \lambda (i) \sum
_{k = 1}^{r} {{\alpha _{k}}(i)
\frac{{{k^{2}}{z^{2}}}}{{2(1 -
kz)}}} \le \lambda (i)\sum_{k = 1}^{r}
{{\alpha _{k}}(i)\frac{
{{k^{2}}{z^{2}}}}{{2(1 - rz)}}} = :\frac{{{\sigma _{i}}{z^{2}}}}{{2(1 -
rz)}}
\quad (z > 0),
\end{equation*}
where ${\sigma _{i}} = \operatorname{Var}Y_{i} = \lambda (i)\sum_{k = 1}^{r} {{k^{2}}{\alpha _{k}}(i)} $. And by applying the
inequality $e^{z}-z-1 \leq \frac{z^{2}}{2}$ for $z<0$, we get
%
\begin{equation}
\label{eq:less0} \log {\mathrm{{E}}} {e^{z({Y_{t}} - {\mu _{i}})}} \le \lambda (i) \sum
_{k = 1}^{r} {{\alpha
_{k}}(i)\frac{{{k^{2}}{z^{2}}}}{2}} = :\frac{{{\sigma _{i}}{z^{2}}}}{2} \quad (z <
0).
\end{equation}
Thus (\ref{eq:less0}) implies $\log {\mathrm{{E}}}{e^{ - z({Y_{t}} -
{\mu _{i}})}} \le \frac{{{\sigma _{i}}{z^{2}}}}{2}$ for $z > 0$.

Therefore, by letting $\tau =1$, $\kappa = 1$, ${\eta _{i}} = \sigma _{i}^{2}$
in Lemma~\ref{lm:Baraud}, we have (\ref{eq:ci1}).
\end{proof}

In the sequel, without employing Cramer-type conditions, we will derive
a Bernstein-type concentration inequality for DCPP with
infinite-dimensional parameters $r =  + \infty $. Theorem~\ref{thm:ci} below is a weighted-sum version that is not the same as Proposition~\ref{thm:cdp}.

\begin{theorem}
\label{thm:ci}
Let $f(x)$ be the measurable function such that $\int _{E} {{f^{2}}(x)}
\lambda (x)\,dx < \infty $, where the $\lambda (x)$ is the intensity
function of the DCPP with parameter $({\alpha _{1}}\int _{E}
{\lambda (x)\,dx}, {\alpha _{2}}\int _{E} {\lambda (x)\,dx} , \ldots )$ under
the condition $\sum_{k = 1}^{\infty }{k{\alpha _{k}}}< \infty $.
Then the concentration inequality for a stochastic integral of DCPP is
\begin{align}
&P \Biggl( {
\int _{E} {f(x)} \Biggl[CP(dx) - \Biggl(\sum
_{k = 1}^{\infty } {k{\alpha _{k}}}
\Biggr) \lambda (x)\,dx \Biggr] \ge \Biggl(\sum_{k = 1}^{\infty }
{k{\alpha _{k}}} \Biggr) \biggl[\sqrt{2y{V_{f}}} +
\frac{y}{3}{{ \| f \|}_{\infty }} \biggr]} \Biggr) \le
{e^{ - y}}\quad (y>0),\label{eq:ie0}
\end{align}
where ${V_{f}} = \int _{E} {{f^{2}}(x)\lambda (x)\,dx}$, ${\| f \|}_{\infty }$ is the supremum of $f(x)$ on $E$.
\end{theorem}

Let $f(x) = {1_{A}}(x)$ in Theorem~\ref{thm:ci}, and we have $P  ( {Y - \mathrm{
{E}}Y \ge \frac{{\mathrm{{E}}Y}}{\lambda }[\sqrt{2y\lambda } +
\frac{y}{3}]}   ) \le {e^{ - y}}$ for $Y \sim \operatorname{DCP}(
{\alpha _{1}}\lambda ,{\alpha _{2}}\lambda , \ldots )$.

From the guide of the proof in Corollary~\ref{thm:ci} below, it is not hard to get
the Theorem~\ref{col:ce} by a linear combination of random variables which
is more important in applications in Sect.~\ref{sec4}. The proof of Theorem~\ref{col:ce} is given in the \hyperref[sec6]{Appendix}. This is essentially from the fact that the infinitely divisible distributions are closed under
the convolutions and scalings.
Corollary~\ref{col:ce} is the version for the sum of $n$ independent random
measures.

\begin{corollary}
\label{col:ce}
For a given set $\{S_{i}\}_{i=1}^n$ in $E$, if we have $n$ independent DCPP $\{ C{P_{i}}(S_{i})\} _{i = 1}^{n}$ with a series of parameters
\begin{equation*}
\{ ({\alpha _{1}}(i)
\smallint{}_{S_{i}} {\lambda (x)\,dx} ,{\alpha _{2}}(i)
\smallint{}_{S_{i}} {\lambda (x)\,dx} , \ldots ) \}
_{i = 1}^{n}
\end{equation*}
correspondingly, then, for a series of measurable function $\{f_i(x)\}_{i=1}^n$, we have
\begin{align}
&P \Biggl( {\sum_{i = 1}^{n}
{ \biggl\{ {
\int _{{S_{i}}} {f_i(x)} \bigl[C {P_{i}}(dx) - {\mu
_{i}}\lambda (x)\,dx \bigr]} \biggr\} } \ge \sum
_{i = 1}^{n} {{\mu _{i}}} \biggl(
\sqrt{2y{V_{i,f_i}}} + \frac{y}{3} {{ \| f_i \|
}_{\infty }} \biggr)} \Biggr)\le {e^{ - ny}}\quad (y>0),\label{eq:ie0A}
\end{align}
where ${\mu _{i}} := \sum_{k = 1}^{\infty }{k{\alpha _{k}}} (i)<
\infty $ and ${V_{i,f_i}} := \int _{{S_{i}}} {{f_i^{2}}(x)} \lambda (x)\,dx<
\infty $.

Moreover, let ${c_{1n}} := \sum_{i = 1}^{n} {{\mu _{i}}} \sqrt{2
{V_{i,f_i}}} $, ${c_{2n}} :=
\sum_{i = 1}^{n} {\frac{{{\mu _{i}}}}{3}}{ \|f_i \|_{\infty }}$, then
\begin{align}
 &P \Biggl( {\sum_{i = 1}^{n}
{ \biggl\{ {
\int _{{S_{i}}} {f_i(x)} \bigl[C {P_{i}}(dx) - {\mu
_{i}}\lambda (x)\,dx \bigr]} \biggr\} } \ge t} \Biggr) \le \exp
\biggl\{ { - n{{ \biggl( {\sqrt{\frac{t}{{{c_{2n}}}} + \frac{
{c_{1n}^{2}}}{{4c_{2n}^{2}}}} -
\frac{{{c_{1n}}}}{{2{c_{2n}}}}} \biggr)} ^{2}}} \biggr\} .\label{eq:f2}
\end{align}
\end{corollary}

The difference between Proposition~\ref{thm:cdp} and
Corollary~\ref{col:ce} is that the DCP random variables in
Proposition~\ref{thm:cdp} have order $r <  + \infty $ while
Corollary~\ref{col:ce} is related to DCP point process (a particular case
of DCPP is the constant intensity $\lambda (A) = \lambda $) with order
$r =  + \infty $. Moreover, Proposition~\ref{thm:cdp} require dependent variance factors $\sigma _{i}^{2}$ for DCP random variables.
\begin{remark}
Significantly, Corollary~\ref{col:ce}
is the weighted-sum version, and it improves the results in \cite{Cleynen14} who deals with the un-weighted sum of NB random
variables (as a special case of DCP distribution). Moreover, \cite{Cleynen14} result relies on Lemma~\ref{lm:Baraud}, which depends
on the sub-Gamma condition. It should be noted that we do not impose any
bounded restriction for $y$ in our concentration inequalities
\eqref{eq:ie0} and \eqref{eq:ie0A}, while Theorem~1 of \cite{Houdre02a} and concentration inequalities of \cite{Houdre02} require that the infinitely divisible random vector
satisfies the Cramer-type conditions and the $y$ in $P(X-EX \geq y)$ is
bounded by a given constant. The moment conditions in Lemma~\ref{lm:Baraud} sometimes are difficult to check, being hard to apply
in the Lasso KKT condition in Sect.~\ref{sec4.1} if the count data are other
infinitely divisible discrete distributions. In our Theorem~\ref{thm:ci} and
Corollary~\ref{col:ce}, we only need to check that the means of the DCP
processes are finite and Theorem~\ref{thm:ci} and Corollary~\ref{col:ce} do not
need the Bernstein-type condition (see Theorem~2.8 in
\cite{Petrov1995} and p.~27 of \cite{Wainwright19}). In addition, our concentration result does not require the higher moment condition proposed by \cite{Kontoyiannis06}: ${\rm E }Z^L = \sum\limits_{k = 1}^\infty  {k^L{\alpha _k}} <\infty~(L>1)$ about the compound distribution $P(Z = i) :=\alpha _{i}$ for the given discrete random variable $Z$.
\end{remark}

Before we give the proof of Corollary~\ref{thm:ci}, we need some notations and a lemma. For a general point process $N(A)$ defined on $\mathbb{R}^{d}$, let
$f$ be any non-negative measurable function on $E$. The Laplace
functional is defined by ${\mathrm{{L}}_{N}}(f) = \mathrm{{E}}[\exp
\{ - \int _{E} f (x)N(dx)\} ]$, it is the stochastic integral is
specified by $\int _{E} f (x)N(dx) = :\sum_{{x_{i}} \in \varPi
\cap E} f ({x_{i}})$, where the $\varPi $ is the state space of the point
process $N(A)$. The Laplace functional for a random measure is a crucial
numerical characteristics that enables us to handle stochastic integral
of the point process. Moreover, we rely on the following theorem due to
Campbell.
\begin{lemma}[See Sect.~3.3 in \cite{Kingman93}]
Let $N(A)$ be a Poisson random measure with intensity $\lambda (t)$ and
a measurable function $f: \mathbb{R}^{d}\rightarrow \mathbb{R}$, the
random sum $S =\sum_{x\in {\varPi }}f(x)$ is absolutely convergent with
probability one if and only if the integral $\int _{\mathbb{R}^{d}}
\min (|f(x)|,1)\lambda (x)\,dx < \infty $. Assume this condition holds,
thus we assert that, for any complex number value $\theta $, the
equation
\begin{equation*}
{\mathrm{{E}}} e^{{\theta S}} = \exp \biggl\{
\int _{{\mathbb{R}^{d}}} { \bigl[ {e^{\theta f(x)}}} - 1 \bigr]\lambda
(x)\,dx \biggr\}
\end{equation*}
holds if the integral on the right-hand side is convergent.
\end{lemma}

\begin{proof}
Now, we can easily compute the DCPP's Laplace functional by Campbell's theorem. Set $C{P_r}(A) =: \sum\limits_{k = 1}^r {k{N_k}(A)}$. And let $\theta=-1$, that is ${{\rm{L}}_N}(f) = \exp \{  \int_E {[{e^{-f(x)}}}  - 1]\lambda (x)dx\} $.  By independence, it follows that
\begin{align*}
{{\rm{L}}_{C{P_r}}}(f) & ={\rm{E}} \exp \{  - \int_E {f(x)} C{P_r}(dx)\} ={\rm{E}} \exp \{  - \int_E {f(x)} \sum\limits_{k = 1}^r  {k{N_k}(dx)} \}  = \prod\limits_{k = 1}^r {\rm{E}} {\exp \{  - \int_E {kf(x)} {N_k}(dx)\} } \\
& = \prod\limits_{k = 1}^r  {\exp \{ \int_E {[{e^{-kf(x)}}}  - 1]{\alpha _k}\lambda (x)dx\}  = } \exp \{  \sum\limits_{k = 1}^r  {\int_E {[{e^{-kf(x)}}}  - 1]{\alpha _k}\lambda (x)dx} \}
\end{align*}
The above Laplace functional of compound random measure seems like L{\'e}vy-Khintchine representation. Define for $\eta>0$
\begin{equation} \label{eq:ie1}
D_r(\eta ) := \exp \left\{ {\{ \int_E {\eta f(x)} [C{P_r}(dx) - (\sum\limits_{k = 1}^r  {k{\alpha _k}} )\lambda (x)]dx\}  - \int_E {\sum\limits_{k = 1}^r  {[{e^{kf(x)}} - k\eta f(x) - 1]{\alpha _k}\lambda (x)} } dx} \right\}.
\end{equation}
Therefore, we obtain ${\rm{E}} D_r(\eta ) = 1$.

It follows from (\ref{eq:ie1}) and Markov's inequality, we have
\begin{align} \label{eq:ie2a}
&~~~~P\left( {\int_E {\eta f(x)} [C{P_r}(dx) - (\sum\limits_{k = 1}^r  {k{\alpha _k}} )\lambda (x)dx] \ge \int_E {\sum\limits_{k = 1}^r  {[{e^{kf(x)}} - k\eta f(x) - 1]{\alpha _k}\lambda (x)} } dx + y} \right)\nonumber\\
& = P(D(\eta ) \ge {e^{ y}}) \le {e^{ - y}}.
 \end{align}
Note that
\begin{align} \label{eq:111}
{e^{k\eta f(x)}} - k\eta f(x) - 1 \le \sum\limits_{i = 2}^\infty  {\frac{{{\eta ^i}}}{{i!}}} \left\| {kf} \right\|_\infty ^{i - 2}{k^2}{f^2}(x) & = {k^2}{f^2}(x)\sum\limits_{i = 2}^\infty  {\frac{{{\eta ^i}}}{{i(i - 1) \cdots 2}}} \left\| {kf} \right\|_\infty ^{i - 2}\nonumber\\
& \le \frac{{{k^2}{\eta ^2}{f^2}(x)}}{2}\sum\limits_{i = 2}^\infty  {{{\left( {\frac{1}{3}k\eta {{\left\| f \right\|}_\infty }} \right)}^{i - 2}}}\nonumber \\
& \le {{\frac{{{k^2}{\eta ^2}{f^2}(x)}}{2}} \mathord{\left/
 {\vphantom {{\frac{{{k^2}{\eta ^2}{f^2}(x)}}{2}} {(1 - \frac{1}{3}k\eta {{\left\| f \right\|}_\infty })}}} \right.
 \kern-\nulldelimiterspace} {(1 - \frac{1}{3}k\eta {{\left\| f \right\|}_\infty })}}.
\end{align}
Hence, using \eqref{eq:111},  (\ref{eq:ie2a}) turns to
\begin{equation} \label{eq:222}
P\left( {\int_E {f(x)} [C{P_r}(dx) - (\sum\limits_{k = 1}^r {k{\alpha _k}} )\lambda (x)dx] \ge \int_E {\sum\limits_{k = 1}^r {{\alpha _k}\left( {\frac{1}{2}\frac{{k^2 \eta {f^2}(x)\lambda (x)}}{{1 - {\textstyle{1 \over 3}}k\eta {{\left\| f \right\|}_\infty }}}dx + \frac{y}{\eta }} \right)} }  } \right) \le {e^{ - y}}.
\end{equation}
Let ${{V_f} = \int_E {{f^2}(x)\lambda (x)dx} }$, notice that
\begin{equation} \label{eq:333}
\begin{split}
\sum\limits_{k = 1}^r {{\alpha _k}\left( {\frac{1}{2}\frac{{\eta {k^2}{V_f}}}{{1 - {\textstyle{1 \over 3}}k\eta {{\left\| f \right\|}_\infty }}} + \frac{y}{\eta }} \right)} & = \sum\limits_{k = 1}^r {{\alpha _k}\left( {\frac{1}{2}\frac{{\eta {k^2}{V_f}}}{{1 - {\textstyle{1 \over 3}}k\eta {{\left\| f \right\|}_\infty }}} + \frac{y}{\eta }(1 - \frac{1}{3}k\eta {{\left\| f \right\|}_\infty }) + \frac{{ky}}{3}{{\left\| f \right\|}_\infty }} \right)} \\
& \ge \sum\limits_{k = 1}^r {{\alpha _k}\left( {k\sqrt {{V_f}y}  + \frac{{ky}}{3}{{\left\| f \right\|}_\infty }} \right)}  = (\sum\limits_{k = 1}^r {k{\alpha _k}} )[\sqrt {{V_f}y}  + \frac{y}{3}{\left\| f \right\|_\infty }].
\end{split}
\end{equation}
Applying \eqref{eq:333}, so by minimizing $\eta $ in \eqref{eq:222}, we have
\begin{equation} \label{eq:ie22}
P\left( {\int_E {f(x)} [C{P_r}(dx) - (\sum\limits_{k = 1}^r {k{\alpha _k}} )\lambda (x)dx] \ge (\sum\limits_{k = 1}^r {k{\alpha _k}} )[\sqrt {2y{V_f}}  + \frac{y}{3}{{\left\| f \right\|}_\infty }]} \right) \le {e^{ - y}}.
\end{equation}
Letting $r \to \infty $ in (\ref{eq:ie22}), then $C{P_r}(A) \xrightarrow{d} CP(A)$, so we finally prove the inequality (\ref{eq:ie0}).
\end{proof}

\section{Application}\label{sec4}

\subsection{Application to KKT conditions}\label{sec4.1}

For negative binomial random variable ${NB}$, the probability mass
function (p.m.f.) is
\begin{equation*}
P({NB} = n) = \frac{{\varGamma (n + r)}}{{\varGamma (r)n!}}{(1 - q)^{s}} {q
^{n}}\quad \bigl(q \in (0,1),n \in \mathbb{N} \bigr),
\end{equation*}
where $s$ is a certain real value number. When $s$ is positive integer, the negative binomial distribution reduces to Pascal distribution which is modeled as the number of failures before the $s$-th success in repeated mutually independent Bernoulli trials (with probability of success $p=1-q$).

The m.g.f. of $NB$ is
\begin{equation}\label{eq:nbd1}
{\varphi _{\mathrm{NB}}}(t) = \exp \biggl\{ \log { \biggl( {
\frac{{1 - q}}{{1 - q{e
^{\mathrm{{i}}t}}}}} \biggr)^{s}} \biggr\} = \exp \Biggl\{ \sum
_{k = 1} ^{\infty }{ - s\log (1 - q)} \cdot
\frac{{{q^{k}}}}{{ - k\log (1 - q)}} \bigl({e^{\mathrm{{i}}kt}} - 1 \bigr) \Biggr\} .
\end{equation}
Here the parametrization of DCP is
\begin{center}
${NB} \sim \operatorname{DCP}({\alpha
_{1}}\lambda ,{\alpha _{2}}\lambda , \ldots )$ with $\lambda = - s
\log (1 - q)$, ${\alpha _{s}} = \frac{{{q^{s}}}}{{ - s\log (1 - q)}}$.
\end{center}

Given the data $\{Y_{i}, {x}_{i}\}_{i=1}^{n}$, the NB regression models
suppose that the responses $\{Y_{i}\}_{i=1}^{n}$ is NB distributed
random variables with mean $\{\mu _{i}\}_{i=1}^{n}$, and the vectors
covariates ${x}_{i}=(x_{i1},\ldots ,x_{ip})^{T} \in
\mathbb{R}^{p}$ are $p$ dimensional fixed (non-random) variable for simplicity; see the so-called NB2 model in \cite{Hilbe11}
for details. Based on the log-linear model in regression analysis, we
usually suppose that $\log \mu _{i} = {f_{0}}({{x}_{i}})$ is a linear
function of ${{x}_{i}}$.

Two types of high-dimensional setting for count data regression are
classified as follows.
\begin{itemize}%
\item
\textit{Sparse approximation for nonparametric regression.}

\noindent Sometimes, it is too rough and too simple to use a linear function to
approximate $\log \mu _{i}$ by a function of ${{x}_{i}}$, denoted as
${f_{0}}({{x}_{i}})$. The connection between $\log \mu _{i}$ and
${{x}_{i}}$ is sometimes unknown, thus it is unlikely to be linear. The
dimension of ${{x}_{i}}$ would be much larger than the sample size when
the $f$ is extremely flexible and complex. In order to capture the
unknown functional relation ${f_{0}}({{x}_{i}})$, we prefer to use a
given dictionary (orthogonal base functions) $\mathbb{D}=\{ {\phi _{j}}(
\cdot )\} _{j = 1}^{p}$ such that the linear combination $\hat{f}(
{{x}_{i}}) = \sum_{j = 1}^{p} {{{\hat{\beta }}_{j}}} {\phi _{j}}(
{x_{ij}})$ is a sparse approximation to ${f_{0}}({{x}_{i}})$ which
contains increasing-dimensional parameter ${\beta ^{*}}: = {(\beta _{1}
^{*}, \ldots ,\beta _{p}^{*})^{T}}$ with $p: = {p_{n}} \to \infty $ as
$n \to \infty $. A~crucial point for orthogonal base functions
$\mathbb{D}$ is that many ${f_{0}}({{x}_{i}})$, which have no sparse
representations in the non-orthogonal base, should be
represented as sparse linear combinations of $\mathbb{D}$ in the
orthogonal scenario; see Sect.~10 of \cite{Hastie15} for details.
In practice, the advantage of the dictionary approach is that the
flexible ${f_{0}}({{x}_{i}})$ may admit a good approximation sparse in
some dictionary functions while not in others. For high-dimensional
count data, \cite{Ivanoff16} mentioned that: ``The richer the
dictionary, the sparser the decomposition, so that $p$ can be larger than
$n$ and the model becomes high-dimensional''. Typical candidates of the
dictionary are Fourier basis, cosine basis, Legendry base, wavelet
basis (Haar basis), etc.
\item
\textit{High-dimensional features in gene engineering}.

\noindent In the big data era, one challenge case we encountered in massive data
sets is that the number of given covariates $p$ is larger than the
sample size $n$ and the responses of interest are measured as counts.
For example, in gene engineering, the covariates (features) are types
of genes impacting on specific count phenotypes. The NB regression is a
flexible framework for the analysis of over-dispersed counts RNA-seq
data; see \cite{Li16}, \cite{Mallick16} for more details.
\end{itemize}

We assume that the expectation of $Y_{i}$ will be related to
${\phi ^{T}({{x}_{i}})\boldsymbol{{\beta }}}$ after a log-transforma\-tion,
\begin{equation*}
{\mu _{i}} := \exp \Biggl\{ \sum_{j = 1}^{p}
{{\beta _{j}}} {\phi _{j}}( {x_{ij}})
\Biggr\} := {e^{\phi ^{T}({{x}_{i}})\boldsymbol{{\beta }}}}:= h({x_{i}}),
\end{equation*}
where $\phi ({{x}_{i}})=(\phi _{1}(x_{i1}),\ldots ,\phi _{p}(x_{ip}))^{T}$
with covariate $\{\phi _{j}(x_{ij})\}$ being the $j$th component of
$\phi ({{x}_{i}})$. Here $\{ {\phi _{j}}( \cdot )\} _{j = 1}^{p}$ are
given transformations of bounded covariances. A trivial example is
${\phi _{j}}(x) = x$.

Let $\{Y_{i}\}_{i=1}^{n}$ be a NB random variable with failure parameters
$\{{q_{i}} = \frac{{{\mu _{i}}}}{{\theta + {\mu _{i}}}}\}$ where
$\theta >0$ is an unknown dispersion parameter which can be estimated (see \cite{Hilbe11}). Then the p.m.f.s of $\{Y_{i}\}_{i=1}^{n}$ are
\begin{equation}
\label{eq:nbd} %
\begin{aligned}[b]
f({y_{i}};{\mu _{i}},
\theta ) &= \frac{{\varGamma (\theta + {y_{i}})}}{
{\varGamma (\theta ){y_{i}}!}}{ \biggl(\frac{{{\mu _{i}}}}{{\theta + {\mu _{i}}}}
\biggr)^{
{y_{i}}}} { \biggl(\frac{\theta }{{\theta + {\mu _{i}}}} \biggr)^{\theta }}
\\
&\propto \exp \bigl\{ {y_{i}}\phi ^{T}({{x}_{i}})
\boldsymbol{\beta } - (\theta + {y_{i}}) \log \bigl(\theta +
e^{\phi ^{T}({{x}_{i}})\boldsymbol{\beta }} \bigr) \bigr\},
\end{aligned}
\end{equation}
where $\mu _{i} >0$ is the parameter with $\mu _{i}=\mathrm{{E}}{Y_{i}}$
and $\operatorname{{Var}} Y_{i}=\mu _{i}+\frac{\mu _{i}^{2}}{\theta } > \mathrm{
{E}}{Y_{i}}$.

The log-likelihood of the predictor variable $\{Y_{i}\}_{i=1}^{n}$ is
%
\begin{equation}
\label{eq:llh} l(\boldsymbol{\beta } )=: \log \Biggl[ \prod
_{i=1}^{n}f \bigl(Y_{i};f_{0}({{x}_{i}}),
\theta \bigr) \Biggr]=\sum_{i=1}^{n}
\bigl[Y_{i}\phi ^{T}({{x}_{i}})\boldsymbol{
\beta } -( \theta +Y_{i})\log \bigl(\theta +e^{\phi ^{T}({{x}_{i}})\boldsymbol{\beta } }
\bigr) \bigr].
\end{equation}

Let $R(y,\beta ,x) = y{\phi ^{T}}(x)\beta - (\theta + y)\log (\theta +
{e^{{\phi ^{T}}(x)\beta }})$ denote the negative average empirical risk
function by $\ell (\boldsymbol{\beta })=- \frac{1}{n} l (\boldsymbol{\beta }):=
\mathbb{P}_{n} R(Y,\beta ,x)$.

We assume that the expectation of $Y_i$ will be related to ${\phi ^T({{x}_i}){\bf{\beta }}}$ after a log-transformation.
$${\mu _i} := \exp \{ \sum\limits_{j = 1}^p {{\beta _j}} {\phi _j}({x_{ij}})\} := {e^{\phi ^T({{x}_i}){\bf{\beta }}}}:= h({X_i}),$$
where $\phi({{x}_i})=(\phi_{1}(x_{i1}),\cdots,\phi_{p}(x_{ip}))^{T}$ with covariate  $\{\phi_{j}(x_{ij})\}$ being the $j$-th component of $\phi({{x}_i})$. Here $\{ {\phi _j}( \cdot )\} _{j = 1}^p$ are given transformations of bounded covariances. A trivial example is ${\phi _j}(x) = x$.

Let $\{Y_i\}_{i=1}^{n}$ be a NB random variable with parameters $\{{q_i} = \frac{{{\mu _i}}}{{\theta  + {\mu _i}}}\}_{i=1}^{n}$ where $\theta>0$ is an unknown dispersion parameter which can be estimated (see \cite{Hilbe11}). Then the p.m.f. of $\{Y_i\}_{i=1}^{n}$ are
\begin{equation}\label{eq:nbd}
\begin{aligned}
f({y_i};{\mu _i},\theta ) &= \frac{{\Gamma (\theta  + {y_i})}}{{\Gamma (\theta ){y_i}!}}{(\frac{{{\mu _i}}}{{\theta  + {\mu _i}}})^{{y_i}}}{(\frac{\theta }{{\theta  + {\mu _i}}})^\theta } \propto \exp \{ {y_i}\phi^{T}({{x}_i})\mathbf{\beta} - (\theta  + {y_i})\log (\theta  + e^{\phi^{T}({{x}_i})\mathbf{\beta}})\} \\
\end{aligned}
\end{equation}
where $\mu _i >0$ is the parameter with $\mu _i={\rm{E}}{Y_i}$ and ${\rm{Var}} Y_i=\mu_i+\frac{\mu_i^{2}}{\theta} > {\rm{E}}{Y_i}$.

The log-likelihood of predictor variable $\{Y_{i}\}_{i=1}^{n}$ is
\begin{equation}\label{eq:llh}
\begin{aligned}
l(\mathbf{\beta} )=:&\log[\prod_{i=1}^{n}f(Y_{i};f_{0}({{x}_i}),\theta)]=\sum_{i=1}^{n}[Y_{i}\phi^{T}({{x}_i})\mathbf{\beta} -(\theta+Y_{i})\log(\theta+e^{\phi^{T}({{x}_i})\mathbf{\beta} })]
\end{aligned}
\end{equation}

Let $R(y,\beta ,x) = y{\phi ^T}(x)\beta  - (\theta  + y)\log (\theta  + {e^{{\phi ^T}(x)\beta }})$, denote the negative average empirical risk function by $\ell (\mathbf{\beta})=- \frac{1}{n} l (\mathbf{\beta}):= \mathbb{P}_n R(Y,\beta ,x)$.

Having obtained high-dimensional covariates, a principal task is to
select important variables. Here, we prefer the weighted $\ell _{1}$-penalized likelihood principle for sparse NB regression,
\begin{equation}
\label{eq:adlasso} \hat{\boldsymbol{\beta }}=\argmin_{\boldsymbol{\beta } \in
\mathbb{R}^{p}} \Biggl
\lbrace \ell (\boldsymbol{ \beta } )+ \sum_{j=1}^{p}
w_{j} |{{\beta _{j}}}| \Biggr\rbrace,
\end{equation}
where $\{ {w_{j}}\} _{j = 1}^{p}$ are data-dependent weights which are
supposed to be specified in the following discussion.

By using KKT conditions (i.e. the first order condition in convex
optimization; see p.~68 of B\"{u}hlmann and van de Geer \cite{Buhlmann11}), the
$\hat{\boldsymbol{\beta }}$ is a solution of (\ref{eq:adlasso}) iff
$\hat{\boldsymbol{\beta }}$ satisfies first order conditions
\begin{equation}
\begin{cases}
- {\dot{\ell
}_{j}}(\hat{\boldsymbol{{\beta }}}) = {w_{j}}
\operatorname{{sign}}( {\hat{\beta }_{j}}) & \mbox{if }
{\hat{\beta } _{j}} \ne 0,
\\
|{\dot{\ell }_{j}}(\hat{\boldsymbol{{\beta }}})|\le {w _{j}} &\mbox{if } {\hat{\beta }_{j}}
= 0,
\end{cases} %
\end{equation}
where ${\dot{\ell }_{j}}(\hat{\boldsymbol{{\beta }}}) := \frac{{\partial
\ell (\hat{\boldsymbol{{\beta }}})}}{{\partial {{\hat{\beta }}_{j}}}}$ ($j =
1,2, \ldots ,p$).

Let $\varPhi ({X})$ be the $n \times p$ dictionary design matrix given by
\[\Phi ({X})  = :{\left( {{\phi _j}({x_{ij}})} \right)_{1 \le i \le n,1 \le j \le p}} = \left( {\begin{array}{*{20}{c}}
{{\phi _1}({x_{11}})}& \cdots &{{\phi _p}({x_{1p}})}\\
 \vdots &{}& \vdots \\
{{\phi _1}({x_{n1}})}& \cdots &{{\phi _p}({x_{np}})}
\end{array}} \right) = :{({\phi ^T}({{{x}}_1}), \cdots ,{\phi ^T}({{{x}}_n}))^T}.\]

The KKT conditions implies $|{\dot{\ell }_{j}}(\hat{\beta })| \le w
_{j}$ for all $j$. And
\begin{align*}
\dot{\ell }(\beta ) &:= \frac{{\partial \ell (\beta )}}{{\partial
\beta }} = -\frac{1}{n} \sum
_{i = 1}^{n} \phi ({{{x}_{i}}})
\biggl[ {Y_{i}} - (\theta + {Y_{i}})
\frac{{{e^{{\phi ^{T}}({{x}_{i}})\beta }}}}{
{\theta + {e^{{\phi ^{T}}({{x}_{i}})\beta }}}} \biggr]
\\
& = -\frac{1}{n} \sum_{i = 1}^{n}
{\frac{{\phi ({{x}_{i}})(
{Y_{i}} - {e^{{\phi ^{T}}({{x}_{i}})\beta }})\theta }}{{\theta + {e
^{{\phi ^{T}}({{x}_{i}})\beta }}}}} = -\frac{1}{n}\sum
_{i = 1} ^{n} {\frac{{\phi ({{x}_{i}})({Y_{i}} - \mathrm{{E}}{Y_{i}})\theta }}{
{\theta + \mathrm{{E}}{Y_{i}}}}}.
\end{align*}

The Hessian matrix of the log-likelihood function is
\begin{equation}
\label{eq:Hessian} \ddot{\ell }(\beta ):= \frac{{{\partial ^{2}}\ell (\beta )}}{{\partial
\beta \partial {\beta ^{T}}}} =
\frac{1}{n}\sum_{i = 1}^{n}
\phi ({{x}_{i}}){\phi ^{T}}({{x}_{i}})
\frac{{\theta (\theta + {Y_{i}})
{e^{{\phi ^{T}}({{x}_{i}})\beta }}}}{{{{(\theta + {e^{{\phi ^{T}}({{x}
_{i}})\beta }})}^{2}}}}.
\end{equation}

Let ${d^{*}} = |\{ j:\beta _{j}^{*} \ne 0\} |$; the true coefficient
vector ${{{\beta }} ^{*}}$ is defined by
\begin{equation}
\label{eq:oracle} {{{\beta }} ^{{{*}}}} = \argmin _{{\beta } \in {{\mathbb{R}}^{p}}} {
\mathrm{{E}}} R(Y,X,{{\beta }}).
\end{equation}
For the optimization problem of maximizing the $\ell _{1}$-penalized
empirical likelihood, our establishment is that KKT conditions are
satisfied for the estimated parameter. But here\vadjust{\goodbreak} we use the true
parameter version of the KKT conditions
\begin{equation}\label{eq:weights}
|{\dot \ell _j}({{\bf{\beta }} ^{\rm{*}}})| =\Biggl|\frac{1}{n}\sum\limits_{i = 1}^n {\frac{{{\phi _j} ({{x}_i})({Y_i} - {\rm{E}}{Y_i})\theta }}{{\theta  + {\rm{E}}{Y_i}}}}\Biggr| \le {w_j}
\end{equation}
to approximate estimated parameter version of KKT conditions.

When $\theta \to \infty$, this
approach of defining data-driven weights ${w_{j}}$ is proposed in \cite{Ivanoff16} for Poisson regression by using concentration inequality for Poisson point process version of Theorem~
\ref{thm:ci}. Similarly, in NB regression the ${w_{j}}$ are determined such that the event
\eqref{eq:weights} holds with high probability for all $j$ by applying Corollary~\ref{col:ce}.

Recall the weighted Poisson decomposition of the $NB$ point processes
\begin{equation*}
{NB}(A) := \sum_{k = 1}^{\infty }{k{N_{k}}(A)},
\end{equation*}
where ${N_{k}}(A)$ is an inhomogeneous Poisson point process with
intensity ${{\alpha _{k}}\int _{A} {\lambda (x)} \,dx}$ and ${\alpha _{k}}
= \frac{{{q^{k}}}}{{ - k\log (1 - q)}}$ defined in \eqref{eq:nbd1}.

Let $\mu (A): = \int _{A} {\lambda (x)} \,dx$, the m.g.f. of ${NB}(A)$ is
written as
\begin{equation}
\label{eq:nbp} {M_{\mathrm{NB}(A)}}(\theta )=\exp \Biggl\{ \sum
_{k = 1}^{\infty }{{\alpha _{k}}\mu (A)
\bigl({e^{ - k\theta }} - 1 \bigr)} \Biggr\}.
\end{equation}

For the subsequent step, let ${{\tilde{\phi }}_{j}}({x_{i}}): = \frac{
{{\phi _{j}}({x_{i}})\theta }}{{\theta + \mathrm{{E}}{Y_{i}}}} \le
{\phi _{j}}({x_{i}})$ and we want to evaluate the complementary events
of the true parameter version of the KKT conditions:
\begin{align}
&P \Biggl( { \Biggl|{\sum_{i = 1}^{n}
{\frac{{{\phi }_{j} ({{{x}}
_{i}})({Y_{i}} - \mathrm{E}{Y_{i}})\theta }}{{\theta + \mathrm{E}
{Y_{i}}}}} } \Biggr|\ge {w_{j}}} \Biggr)= P
\Biggl( { \Biggl| \sum_{i = 1}^{n}
{{\tilde{\phi }}_{j} ({{{x}}_{i}})
({Y_{i}} - E{Y_{i}})} \Biggr|\ge
{w_{j}}} \Biggr) \quad \text{for } j = 1,2, \ldots ,p.\label{eq:kkt}
\end{align}
The aim is to find the values $\{ {w_{j}}\} _{j = 1}^{p}$ by the
concentration inequalities we proposed. Then it is sufficient to
estimate the probability in the right-hand side of the above inequality.

Let ${d^*} = |\{ j:\beta _j^* \ne 0\} |$, the true coefficient vector ${{\bf{\beta }} ^*}$ is defined by
\begin{equation}\label{eq:oracle}
{{{\beta }} ^{\rm{*}}}{\rm{ = }}\mathop {\rm{argmin}}\limits_{{\beta}  \in {{\mathbb{R}}^p}} {\rm{E}} R(Y,X,{{\beta }}).
\end{equation}
For the optimization problem of maximizing the $\ell_{1}$-penalized empirical likelihood, our establishment is that KKT conditions are satisfied for the estimated parameter. But here we use the true parameter version of KKT conditions
\begin{equation}\label{eq:weights}
|{\dot \ell _j}({{\bf{\beta }} ^{\rm{*}}})| =|\frac{1}{n}\sum\limits_{i = 1}^n {\frac{{{\phi _j} ({{x}_i})({Y_i} - {\rm{E}}{Y_i})\theta }}{{\theta  + {\rm{E}}{Y_i}}}}| \le {w_j}
\end{equation}
to approximate estimated parameter version of KKT conditions. This approach of defining data-driven weights ${w_j}$ is proposed in \cite{Ivanoff16}, we want to find ${w_j}$ such that the event \eqref{eq:weights} hold with high probability for all $j's$.

Recall the weighted Poisson decomposition of the NB point processes
\[NB(A) := \sum\limits_{k = 1}^\infty  {k{N_k}(A)} \]
where ${N_k}(A)$ is inhomogeneous Poisson point process with intensity ${{\alpha _k}\int_A {\lambda (x)} dx}$ and ${\alpha _k} = \frac{{{p^k}}}{{ - k\log (1 - p)}}$.

 Let $\mu (A): = \int_A {\lambda (x)} dx$, the m.g.f. of $NB(A)$ is written as
\begin{equation} \label{eq:nbp}
{M_{NB(A)}}(\theta )=\exp \{ \sum\limits_{k = 1}^\infty  {{\alpha _k}\mu (A)({e^{ - k\theta }} - 1)} \}.
\end{equation}

For the subsequent step, let ${{\tilde \phi }_j}({x_i}): = \frac{{{\phi _j}({x_i})\theta }}{{\theta  + {\rm{E}}{Y_i}}} \le {\phi _j}({x_i})$ and we want to evaluate the complementary events of true parameter version of KKT conditions:
\begin{equation} \label{eq:kkt}
P\left( {\left| {\sum\limits_{i = 1}^n {\frac{{{\phi}_j ({{\bf{x}}_i})({Y_i} - {\rm E}{Y_i})\theta }}{{\theta  + {\rm E}{Y_i}}}} } \right| \ge {w_j}} \right) = P\left( {|\sum\limits_{i = 1}^n {{\tilde \phi }_j ({{\bf{x}}_i})({Y_i} - {\rm E}{Y_i})} | \ge {w_j}} \right) ~~\text{for}~~j = 1,2, \cdots ,p.
\end{equation}
The aim is to find the value $\{ {w_j}\} _{j = 1}^p$ by the concentration
inequalities we proposed. Then it is sufficient to estimate the probability in the right hand of the above inequality.

Let $\mu ({A})$ be a Lebesgue measure for $A \in \mathbb{R}^d$ and $N_{i,k}({A})$ be Poisson random measure with rate $\frac{{{\alpha _k}(i)\mu ({A})}}{{{\mu _i}}}$, where ${\mu _i} = \sum\limits_{k = 1}^\infty  {k{\alpha _k}(i)}$ with ${\alpha _k}(i) := \frac{{q_i^k}}{{ - k\log (1 - {q_i})}}$ and $\{{q_i} := \frac{{{\mu _i}}}{{\theta  + {\mu _i}}}\}_{i=1}^{n}$. We define the independent NB point processes $NB_{i}(A)$ as :
$$NB_{i}(A) =: \sum\limits_{k = 1}^{\infty} {k{N_{i,k}}(A)},~~\text{with}~~{\rm{E}}[NB_{i}({A})] = \sum\limits_{k = 1}^\infty  k \frac{{{\alpha _k}(i)\mu ({A})}}{{{\mu _i}}} = \mu ({A})$$
where $\{{N_{i,k}}(A)\}$ are independent for each $k, i$.

For ${\left[ {0,1} \right]^d}:= \cup _{i = 1}^n{S_i}$ being a disjoined union, we assume that the intensity function
$$\lambda (t) = \sum\limits_{i = 1}^n {\frac{{h({X_i})}}{{\mu ({S_i})}}} {{\mathrm{I}}_{{S_i}}}(t)$$
is a histogram type estimator (\cite{Baraud2009}). It is just a piece-wise constant intensity function.

Now we will apply Corollary~\ref{col:ce}, note that
$$\mu ({S_i}) = \int_{{S_i}} {\lambda (t)d} t ={h({x_i})}:={\rm{E}} Y_i$$
via the definition of $\lambda (t)$ that we defined before.

Thus, both $(N{B}_1({S_1}), \ldots ,N{B}_n({S_n}))$ and $({Y_1}, \ldots ,{Y_n})$ have the same law by the weighted Poisson decomposition.

Choosing $f(x)={g_j}(x) = \sum\limits_{l = 1}^n {{{\tilde \phi }_j}({{{x}}_l}){{\rm{1}}_{{S_l}}}(x)} $, it derives
\[\sum\limits_{i = 1}^n {\int_{{S_i}} {g_j(x)} NB_i(dx) = } \sum\limits_{i = 1}^n {\int_{{S_i}} {\sum\limits_{l = 1}^n {{{\tilde \phi }_j}({{{x}}_l}){{\rm{I}}_{{S_l}}}(x)} } NB_i(dx) = \sum\limits_{i = 1}^n {{{\tilde \phi }_j}({{{x}}_i}){Y_i}} } \]
for each $j$.

Assume that there exists constants $C_{1},C_{2}$ such that
\begin{equation}\label{eq:c1c2}
{C_1} =\mathop {\max }\limits_{1 \le i \le n}\mu_i = \mathop {\max }\limits_{1 \le i \le n} \sum\limits_{k = 1}^\infty  {k{\alpha _k}(i)}  = O(1)~\text{and}~O(1) = \mathop {\max }\limits_{1 \le i \le n} h({x_i}) \le {C_2}.
\end{equation}
Then
\begin{align*}
{V_{i,{g_j}}} = \int_{{S_i}} {g_j^2(x)} \lambda (x)dx &= \int_{{S_i}} {\sum\limits_{l = 1}^n {\tilde \phi _j^2({x_l}){{\rm{I}}_{{S_l}}}(x)} } \sum\limits_{i = 1}^n {\frac{{h({x_i})}}{{\mu ({S_i})}}} {{\rm{I}}_{{S_i}}}(t)dx\\
& = \tilde \phi _j^2({x_i})h({x_i}) \le {C_2}\mathop {\max}\limits_{1 \le i \le n} \tilde \phi _j^2({x_i})\le {C_2}\mathop {\max}\limits_{1 \le i \le n} \phi _j^2({x_i}),
\end{align*}
$${ \|{{g_{j}}(x)}  \|_{\infty }}\le \mathop {\max }\limits_{i = 1,...,n} \left| {{\phi _j}({{{x}}_i})} \right|.$$
From (\ref{eq:kkt}), we apply concentration inequality in Corollary~\ref{col:ce} to the last probability in inequality below
\begin{gather*}
P \Biggl( {\frac{1}{n} \Biggl|{\sum
_{i = 1}^{n} {{{\tilde{\phi }}
_{j}}({\mathbf{{x}}_{i}}) ({Y_{i}} -
{\rm E}{Y_{i}})} } \Biggr| \ge {C_{1}}\biggl[\sqrt{2y
{C_{2}\mathop {\max}\limits_{1 \le i \le n} \phi _j^2({x_i})}} + \frac{y}{3}{{\|
{{g_{j}}(x)} |}_{\infty
}}\biggr]} \Biggr)
\\
\quad \le P \Biggl( { \Biggl| {\frac{1}{n}\sum
_{i = 1}^{n} { \biggl\{ {
\int _{{S_{i}}} {g(x)} \bigl[{NB}_i(dx) - {\mu
_{i}}\lambda (x)\,dx\bigr]} \biggr\} } } \Biggr|\ge
\frac{1}{n}\sum_{i = 1}^{n} {{
\mu _{i}}\biggl[\sqrt{2y {V_{i,{g_{j}}}}} +
\frac{y}{3}{{ \| {{g_{j}}(x)} \| }
_{\infty }}\biggr]} } \Biggr).
\end{gather*}
Now we can let ${w_{j}} = {C_{1}}[\sqrt{2y{C_{2}\mathop {\max}\limits_{1 \le i \le n} \phi _j^2({x_i})}} +
\frac{y}{3} \mathop {\max }\limits_{i = 1,...,n} \left| {{\phi _j}({{{x}}_i})} \right|]$.

From (\ref{eq:kkt}), we apply concentration inequality in Corollary~\ref{col:ce} to the last probability in above inequality. This implies
\[P\left( {\frac{1}{n}\left| {\sum\limits_{i = 1}^n {\frac{{{\phi _j}({{\bf{x}}_i})({Y_i} - {\rm E}{Y_i})\theta }}{{\theta  + {\rm E}{Y_i}}}} } \right| \ge {w_j}} \right) \le P\left( {\left| {\frac{1}{n}\sum\limits_{i = 1}^n {\left\{ {\int_{{S_i}} {g_j(x)} [NB_i(dx) - {\mu _i}\lambda (x)dx]} \right\}} } \right| \ge {w_j}} \right) \le 2{e^{ - ny}}\]
Take $y = \frac{\gamma }{n}\log p$ where $\gamma >0$, we have
\begin{equation} \label{eq:pe}
P\left( {\frac{1}{n}\left| {\sum\limits_{i = 1}^n {\frac{{{\phi _j}({{\bf{x}}_i})({Y_i} - {\rm E}{Y_i})\theta }}{{\theta  + {\rm E}{Y_i}}}} } \right| \ge {w_j}} \right) \le \frac{2}{{{p^\gamma }}}.
\end{equation}
The weights is proportional to
\[{w_j} \propto \sqrt {\frac{{2\gamma \log p}}{n}} C_1(\mathop {\max }\limits_{i = 1,...,n} \left| {{\phi _j}({{\bf{x}}_i})} \right|) + \frac{{\gamma \log p}}{{3n}}\mathop {\max }\limits_{i = 1,...,n} \left| {{\phi _j}({{\bf{x}}_i})} \right|.\]
Here the $\gamma$ can be treated as tuning parameters. In high-dimensional statistics, we often have dimension assumption $\frac{{\log p}}{{n}}=o(1),\mathop {\max }\limits_{i = 1,...,n} \left| {{\phi _j}({{\bf{x}}_i})} \right|<\infty$, the term $\frac{{\gamma \log p}}{{3n}}\mathop {\max }\limits_{i = 1,...,n} \left| {{\phi _j}({{\bf{x}}_i})} \right|$ is negligible in the weights, comparing to the term $\sqrt {\frac{{2\gamma \log p}}{n}}$.

Then for large dimension $p$, the above upper bound of the probability inequality tends to zero. Therefore, with high probability, the event of the KKT condition holds.

Based on (\ref{eq:pe}) and some regular assumptions such as
compatibility factor (a generalized restricted eigenvalue condition),
we will derive the following $\ell _{1}$-estimation error with high
probability:
\begin{equation*}
\bigl\|{\hat{\beta }- {\beta ^{*}}} \bigr\|
_{1} \le O_{p} \biggl({d^{*}}\sqrt{
\frac{
{\log p}}{n}} \biggr)
\end{equation*}
in the next section.

Under the restricted eigenvalue condition proposed by \cite{Bickel09}, oracle inequalities for a weighted Lasso regularized Poisson binomial regression are studied in \cite{Ivanoff16}. Since
the negative binomial is neither sub-gaussian nor a member of the
exponential family, our extension is not straightforward. We employ the
assumption of compatibility factor which renders the proof simpler than
that of restricted eigenvalue condition.

\subsection{Oracle inequalities for the weighted Lasso in count data regressions}
\label{sec4.2}

The ensuing subsections are described in two steps: The first step is
to construct two lemmas to get a lower or upper bound for symmetric
Bregman divergence in the weighted Lasso case. Adopting a symmetric
Bregman divergence, the second step is to derive oracle inequalities by
some inequality scaling tricks.

\subsubsection{Case of NB regression}\label{nb}

The symmetric Bregman divergence (sB divergence) is defined by
\begin{equation*}
{D^{s}}(\hat{\beta },\beta ) = {(\hat{\beta }- \beta
)^{T}} \bigl( \dot{\ell }(\hat{\beta }) - \dot{\ell }(\beta )
\bigr);
\end{equation*}
see Nielsen and Nock \cite{Nielsen09}.

The following lemma is a crucial inequality for deriving the oracle
inequality under the assumption of a compatibility factor which is also
used in Yu \cite{Yu2010} for the Cox model. Let ${\beta ^{*}}$ be the
true estimated coefficients vector defined by \eqref{eq:oracle}, and
$H = \{ j:\beta _{j}^{*} \ne 0\} $, ${H^{c}} = \{ j:\beta _{j}^{*} = 0\}$.
Then we have the following.

\begin{lemma}
\label{lem:ulb}
Let $\hat{\beta }$ be the estimate in (\ref{eq:adlasso}), set
${l_{m}} := {\| {\dot{l}({\beta ^{{{*}}}})} \|_{\infty }}$, and
$\tilde{\beta }= \hat{\beta }- {\beta ^{*}}$. Then, we have
\begin{equation*}
({w _{\min }} - {l_{m}}){ \|{{{\tilde{\beta
}}_{{H^{c}}}}} \| _{1}} \le {D^{s}}
\bigl( \hat{\beta },\beta ^{*} \bigr) + ({w _{\min }} -
{l_{m}}){ \| {{{\tilde{\beta }}_{{H^{c}}}}} \|
_{1}} \le ({w _{\max }} + {l_{m}}) {
\| {{{\tilde{\beta }}_{H}}} \| _{1}},
\end{equation*}
where ${w_{\max }} =  \max_{1 \le j \le p} {w_{j}}$, ${w_{\min }} =  \min_{1 \le j \le p} {w_{j}}$.

Moreover, in the event ${l_{m}} \le \frac{{\xi - 1}}{{\zeta + 1}}
{w_{\min }}$ for any $\xi > 1$ and ${\| {\varPhi ({{X}})} \|_{\infty }} \le K$, with probability $1-{p^{1 - O(1){r
^{2}}}}$, we have
\begin{equation}
\label{eq:ulb} {\| {{{\tilde{\beta }}_{{H^{C}}}}}
\|_{1}} \le \frac{{\xi {w _{\max }}}}{
{{w _{\min }}}}{ \|{{{\tilde{\beta
}}_{H}}} \|_{1}}.
\end{equation}
\end{lemma}

\begin{proof}
The left inequality is obtained by ${D^{s}}(\hat{\beta },{\beta ^{{
{*}}}}) = {(\hat{\beta }- {\beta ^{{{*}}}})^{T}}[\dot{\ell }(
\hat{\beta }) - \dot{l}({\beta ^{{{*}}}})] \ge 0$ due to convexity
of the function $\ell (\beta )$. Note that
\begin{align*}
{{\tilde{\beta }}^{T}} \bigl\{ \dot{\ell } \bigl({\beta
^{{{*}}}} + \tilde{\beta } \bigr) - \dot{\ell } \bigl({\beta
^{{{*}}}} \bigr) \bigr\} &= \sum_{j \in {H^{c}}}
{{{ \tilde{\beta }}_{j}}} { \dot{\ell }_{j} \bigl(
{\beta ^{{{*}}}} + \tilde{\beta } \bigr)} + \sum
_{j \in H} {{{\tilde{\beta }}_{j}}}{ \dot{\ell
}_{j} \bigl({ \beta ^{{{*}}}} + \tilde{\beta } \bigr)} +
{{\tilde{\beta }}^{T}} \bigl( - \dot{\ell } \bigl({\beta
^{{
{*}}}} \bigr) \bigr)
\\
 & \le \sum_{j \in {H^{c}}} { -
{w_{j}} {{\hat{\beta }}_{j}}} \operatorname{{sign}}({{
\hat{\beta }}_{j}}) + \sum_{j \in H} {w
_{j}} { | {{\tilde{\beta }}_{j}} |
} + \| \tilde{\beta } \| _{1}{l_{m}}
\\
& = \sum_{j \in {H^{c}}} { - {w _{j}} |
{{{\hat{\beta }}_{j}}} |  } + \sum
_{j \in H} {{w _{j}} | {{{\tilde{\beta
}}_{j}}} | } + {l_{m}}
\|  {{{\tilde{\beta }}_{H}}} \| _{1} +
{l_{m}} \| {{{\tilde{\beta }}_{{H^{c}}}}} \|
_{1}
\\
& \le \sum_{j \in {H^{c}}} { - {w _{\min }}
|  {{{\hat{\beta }} _{j}}} | } + \sum
_{j \in H} {{w _{\max }} | {{{\tilde{
\beta }} _{j}}} | } + {l_{m}}
\| {{{\tilde{\beta }}_{H}}} \|_{1} +
{l_{m}} \|  {{{\tilde{\beta }}_{{H^{c}}}}} \|
_{1}
\\
&= ({l_{m}} - {w _{\min }}) \|  {{{\tilde{\beta
}}_{{H^{c}}}}} \| _{1} + ({w _{\max }} +
{l_{m}}) \|  {{{\tilde{\beta }}_{H}}} \|
_{1}.
\end{align*}
where the second last inequality follows from KKT conditions and dual norm inequality.

Thus we obtain (\ref{eq:ulb}). In fact, inequality (\ref{eq:ulb}) turns
to
%
\begin{equation}
\label{eq:ie2} \frac{{2{w _{\min }}}}{{\xi + 1}} \|  {{{\tilde{\beta
}}_{{H^{c}}}}} \| _{1} \le {D^{s}} \bigl(
\hat{\beta },\beta ^{*} \bigr) + \frac{{2{w _{\min }}}}{
{\xi + 1}} \|  {{{
\tilde{\beta }}_{{H^{c}}}}} \| _{1} \le
\frac{{2
\xi {w _{\max }}}}{{\xi + 1}} \|  {{{\tilde{\beta }}_{H}}} \|
_{1}
\end{equation}
due to ${w _{\min }} - {l_{m}} \ge \frac{{2{w _{\min }}}}{{\xi + 1}}$,
${w _{\max }} + {l_{m}} \le \frac{{2\xi }}{{\xi + 1}}{w _{\max }}$.

For the event $\{ \|\dot{\ell }({\beta ^{{{*}}}})|{|_{\infty }}
\le \frac{{\xi - 1}}{{\xi + 1}}{w_{\min }}\}$, applying the NB
concentration inequality \eqref{eq:f2}, we have
\begin{align*}
&P \biggl( { \bigl\|  \dot{\ell } \bigl({\beta ^{{{*}}}}
\bigr) \bigr\|  _{\infty } \ge \frac{
{\xi - 1}}{{\xi + 1}}{w_{\min }}}
\biggr)
\\*
&\quad  \le pP \Biggl( { \Biggl|  {\sum_{i = 1}^{n}
{\frac{{{\phi _{1}}({x_{i}})\theta }}{{(
\theta + E{Y_{i}})}}({Y_{i}} - E{Y_{i}})} } \Biggr
|  \ge \frac{{\xi -
1}}{{\xi + 1}}n{w_{\min }}} \Biggr)
\\
&\quad  \le p\exp \biggl\{ { - n{{ \biggl( {\sqrt{\frac{{\xi - 1}}{{\xi +
1}} \cdot
\frac{{n{w_{\min }}}}{{{c_{2n}}}} + \frac{{c_{1n}^{2}}}{
{4c_{2n}^{2}}}} - \frac{{{c_{1n}}}}{{2{c_{2n}}}}}
\biggr)}^{2}}} \biggr\}
\\
&\quad  = p\exp \biggl\{ { - n{{ \biggl( {{\frac{{\xi  - 1}}{{\xi  + 1}} \cdot \frac{{n{w_{\min }}}}{{{c_{2n}}}}} \mathord{\left/
 {\vphantom {{\frac{{\xi  - 1}}{{\xi  + 1}} \cdot \frac{{n{w_{\min }}}}{{{c_{2n}}}}} {\sqrt {\frac{{\xi  - 1}}{{\xi  + 1}} \cdot \frac{{n{w_{\min }}}}{{{c_{2n}}}} + \frac{{c_{1n}^2}}{{4c_{2n}^2}}}  + \frac{{{c_{1n}}}}{{2{c_{2n}}}}}}} \right.
 \kern-\nulldelimiterspace} {\sqrt {\frac{{\xi  - 1}}{{\xi  + 1}} \cdot \frac{{n{w_{\min }}}}{{{c_{2n}}}} + \frac{{c_{1n}^2}}{{4c_{2n}^2}}}  + \frac{{{c_{1n}}}}{{2{c_{2n}}}}}}\biggr)}^{2}}} \biggr\}
\\
&\quad  \sim p\exp \bigl\{ { - nw_{\min }^{2}O(1)} \bigr\} \sim
{p ^{1 - O(1){r^{2}}}}.
\end{align*}
where the second last $\sim$ is due to $c_{1n}=O(n)=c_{2n}$ by the assumption \eqref{eq:c1c2} and the the last $\sim$ is by letting ${w_{\min }} = O(r\sqrt{\frac{{\log p}}{n}})$.

Here $r>0$ is a
suitable tuning parameter, such that the event $\{ \|\dot{\ell }(
{\beta ^{{{*}}}})|{|_{\infty }} \le \frac{{\xi - 1}}{{\xi + 1}}
{w_{\min }}\}$ holds with probability tending to 1 as $p,n \to \infty$.
\end{proof}

In order to derive Theorem~\ref{theo-oi}, we adopt a crucial inequality
by using the similar approach in Lemma~5 of \cite{Zhang17}, which
is the technique of Taylor's expansion for convex log-likelihood
functions.

\begin{lemma}\label{lem:ie3}
 Let the Hessian matrix $\ddot \ell(\beta)$ be given in (\ref{eq:Hessian}), and for $\delta \in {\mathbb R}^p$, assume the
identifiability condition: ${\phi ^{T}}({{x}_{i}})(\beta + \delta )={\phi ^{T}}({{x}_{i}})\beta $ implies ${\phi ^{T}}({{x}_{i}})\delta  = 0$, then we have
\[{D^s}(\beta  + \delta,\beta ) \ge {\delta^T}\ddot \ell(\beta )\delta{e^{ - 2{{|| { \Phi ({{X}})} ||}_\infty }|| \delta ||}}\]
where ${|| {\Phi ({X})} ||_\infty } = \max \{ | {\phi _j({x_{ij}})} |;1 \le i \le n,1 \le j \le p\} .$
\end{lemma}

By the result in the second parts of Theorem~\ref{lem:ulb}, it says that in the event $\left\{ {{{|| {\dot \ell({\beta ^{\rm{*}}})} ||}_\infty } \le \frac{{\xi  - 1}}{{\xi  + 1}}{\lambda _{\min }}} \right\}$, the error of estimate $\tilde \beta  = \hat \beta  - {\beta ^*}$ belongs to the following cone set:
\begin{equation}\label{eq:re}
{\mathop{\rm S}\nolimits} (\eta ,H) = \{ {b} \in {\mathbb{R}^p}:{|| {{{b}_{{H^c}}}} ||_1} \le \eta {|| {{{b}_H}} ||_1}\}
\end{equation}
with $\eta= \frac{{\xi {w_{\max }}}}{{{w_{\min }}}}$.

Let ${d^{*}(\beta ^*)} = :\left| {H(\beta ^*)} \right|$ with $H(\beta ^*) = :\{ j:{\beta ^*_j} \ne 0,~\beta ^* = ({\beta ^*_1}, \cdots ,{\beta ^*_j}, \cdots ,{\beta ^*_p}) \in {\mathbb{R}^p}\} $. Sometimes, we write $H(\beta ^*)$ as $H$ and $d^{*}(\beta ^*)$ as $d^{*}$ if there is no ambiguity. Define compatibility factor (see \cite{Geer2007}) of a $p \times p$ nonnegative-definite matrix $\Sigma$ as
\[C(\eta ,H,\Sigma ) = \mathop {\inf }\limits_{0 \ne b \in {\rm{S}}(\eta ,H)} \frac{{{({d^o(b)})^{1/2}}{{({b^T}\Sigma b)}^{1/2}}}}{{{{\| {{b_H}} \|}_1}}} > 0,~~(\eta  \in \mathbb{R}).\]
where ${\rm{S}}(\eta ,H)=\{ b \in {\mathbb{R}^p}:{\| {{b_{{H^c}}}} \|_1} \le \eta {|| {{b_H}} ||_1}\}$ is the cone condition.

The compatibility factor is strongly linked to the restricted eigenvalue  proposed by \cite{Bickel09},
\begin{equation} \label{eq:Bickel09}
RE(\eta ,H,\Sigma ) = \mathop {\inf }\limits_{0 \ne {{b}} \in {\rm{S}}(\eta ,H)} \frac{{{{({{{b}}^T}\Sigma {{b}})}^{1/2}}}}{{{\left\| b \right\|_2}}} > 0 ,(\eta  \in {\mathbb{R}}).
\end{equation}
Actually, by Cauchy-Schwarz inequality, it is evident that ${[C(\eta ,H,\Sigma )]^2} \ge {[RE(\eta ,H,\Sigma )]^2}$. In the proof of Theorem~\ref{theo-oi}, the upper bounds of oracle inequalities are sharper if we use a compatibility factor rather than restricted eigenvalue.

Next, let $C(\eta ,H) = C(\xi ,H,\ddot l({\beta ^*}))$, then we have the following theorem which is analogous to \cite{Zhang17} who gives oracle inequalities for the Elastic-net estimate in NB regression.
\begin{theorem}\label{theo-oi}
Suppose ${|| {\Phi ({{X}})} ||_\infty } \le K$ with some $K > 0$ and let $\tau  = \frac{{K(\xi  + 1){d^*}w_{\max }^2}}{{2{w_{\min }}{{\left[ {C(\xi {w_{\max }}/{w_{\min }},H)} \right]}^2}}} \le {\frac{1}{2}e^{ - 1}}$ with assumption \eqref{eq:c1c2}. Setting ${w_{\min }} = O(r\sqrt {\frac{{\log p}}{n}})$ where $r>0$ is a suitable tuning parameter, then in the event $\{ ||\dot \ell ({\beta ^{\rm{*}}})|{|_\infty } \le \frac{{\xi  - 1}}{{\xi  + 1}}{w_{\min }}\}$ with probability tending to 1 as $p,n \to \infty$, we have
\[{|| {\hat \beta  - {\beta ^{\rm{*}}}} ||_1} \le \frac{{{e^{{2a_\tau}}}(\xi  + 1){d^*}w_{\max }^2}}{{2{w_{\min }}{{\left[ {C(\xi {w_{\max }}/{w_{\min }},H)} \right]}^2}}}={O_p}({d^*}\frac{{w_{\max }^2}}{{{w_{\min }}}})\]
where ${a_\tau } \le \frac{1}{2}$ is smaller a solution of the equation $a {e^{ - 2a }} = \tau$.
And
\[{D^s}(\hat \beta ,{\beta ^*}) \le \frac{{4{e^{2{a_\tau }}}{\xi ^2}{d^*}w_{\max }^2}}{{{{(\xi  + 1)}^2}{{\left[ {C(\xi {w_{\max }}/{w_{\min }},H)} \right]}^2}}}\]
\end{theorem}
\begin{proof}
Let $\tilde \beta  = \hat \beta  - {\beta ^*} \ne 0$ and
$b = \tilde \beta /{|| {\tilde \beta } ||_1}$, and observe that $ \ell({\beta ^*} + {b}x)$ is a convex function in $x$ since $\ell(\beta )$ is convex. From the proof of (\ref{eq:ie2}) in Lemma~\ref{lem:ulb}, we have
\begin{equation} \label{eq:ie-sbd}
{{b}^T}[\dot l({\beta ^*} + {b}x) - \dot \ell({\beta ^*})] + \frac{{2{w_{\min }}}}{{\xi  + 1}}{|| {{{b}_{{H^C}}}} ||_1} \le \frac{{2\xi {w_{\max }}}}{{\xi  + 1}}{|| {{{b}_H}} ||_1}
\end{equation}
holding for $x \in [0,{|| {\tilde \beta } ||_1}]$.

Since ${\ell(\beta )}$ is a convex in $\beta$, then the function ${{b}^T}[\dot \ell(\beta  + {b}x) - \ell(\beta )]$ is increasing in $x$.) and ${b} \in {\rm{S}}(\xi {w_{\max }}/{w_{\min }},H)$.
 By the assumption, we have ${|| { \Phi ({X}) } ||_\infty }|| {x{b}} ||_1 \le Kx$ by noticing that $|| {b} ||_1 = 1$. And assuming ${|| {\Phi ({{x}})} ||_\infty } \le K$), we obtain following by the Lemma~\ref{lem:ie3}
\[
x{b^T}[\dot \ell({\beta ^*} + {b}x) - \dot \ell({\beta ^*})] \ge {x^2}{e^{ - 2{{|| { \Phi ({{X}})} ||}_\infty }|| {x{b}} ||}}{{b}^T}\ddot \ell(\beta ){b} \ge {x^2}{e^{ - 2Kx}}{{b}^T}\ddot \ell(\beta ){b},
\]
this implies
\begin{equation} \label{eq:ie4}
{{b}^T}[\dot \ell({\beta ^*} + {b}x) - \dot \ell({\beta ^*})] \ge x{e^{ - 2Kx}}{{b}^T}\ddot \ell(\beta ){b}.
\end{equation}

For the Hessian matrix evaluated at the true coefficient ${{\beta ^*}}$, the compatibility factor is written as
${\rm{C}}(\eta ,H)= :{\rm{ C}}(\eta ,H,\ddot \ell ({\beta ^*}))$. From the definition of compatibility factor and the inequality \eqref{eq:ie4}, we have
\begin{align*}
Kx{e^{ - 2Kx}}{\left[ {C(\xi {w_{\max }}/{w_{\min }},H)} \right]^2}{|| {{{b}_H}} ||^2}/{d^*} & \le Kx{e^{ - 2Kx}}{{b}^T}\ddot \ell(\beta ){b} \\
(\text{by \eqref{eq:ie4}})~~ & \le K{{b}^T}[\dot \ell({\beta ^*} + {b}x) - \dot \ell({\beta ^*})] \\
(\text{by~(\ref{eq:ie2})})~~ & \le K(\frac{{2\xi {w_{\max }}}}{{\xi  + 1}}{|| {{{b}_H}} ||_1} - \frac{{2{w_{\min }}}}{{\xi  + 1}}{|| {{{b}_{{H^c}}}} ||_1})\\
& = K[\frac{{2\xi {w_{\max }}}}{{\xi  + 1}}{|| {{{b}_H}} ||_1} - \frac{{2{w_{\min }}}}{{\xi  + 1}}(1 - {|| {{{b}_H}} ||_1})] \\
& \le K[\frac{2}{{\xi  + 1}}(\xi {w_{\max }} + {w_{\max }}){|| {{{b}_H}} ||_1} - \frac{{2{w_{\min }}}}{{\xi  + 1}}] \\
& = K[2{w_{\max }}{|| {{{b}_H}} ||_1} - \frac{{2\xi {w_{\min }}}}{{\xi  + 1}}] \\
& \le \frac{{K(\xi  + 1)|| {{{b}_H}} ||_1^2w_{\max }^2}}{{2{w_{\min }}}}
\end{align*}
where the last inequality is due to $\frac{{2{w_{\min }}}}{{\xi  + 1}} + {\rm{ }}\frac{{(\xi  + 1){{|| {{{b}_H}} ||}_1}w_{\max }^2}}{{2{w_{\min }}}} \ge 2{w_{\max }}{|| {{{b}_H}} ||_1}$.

Then we have
\begin{equation} \label{eq:ie5}
Kx{e^{ - 2Kx}} \le \frac{{K(\xi  + 1){d^*}w_{\max }^2}}{{2{w_{\min }}{{\left[ {C(\xi {w_{\max }}/{w_{\min }},H)} \right]}^2}}} = :\tau
\end{equation}
for any $x \in [0,{|| {\tilde \beta } ||_1}]$.

So with constrain $x \in [0,{|| {\tilde \beta } ||_1}]$, let ${a_\tau },{b_\tau },({a_\tau }<0.5<b_\tau )$ be the two solutions of the equation $\{z: z {e^{ - 2z}} = \tau\} $. Let $S = [0,{|| {\tilde \beta } ||_1}] \cap \{ ( - \infty ,{a_\tau }] \cup [{b_\tau }, + \infty )\}$ satisfy \eqref{eq:ie5}, see Figure 1.

\begin{figure}[!ht]
\centering
\includegraphics[height=0.4\textwidth=0.8]{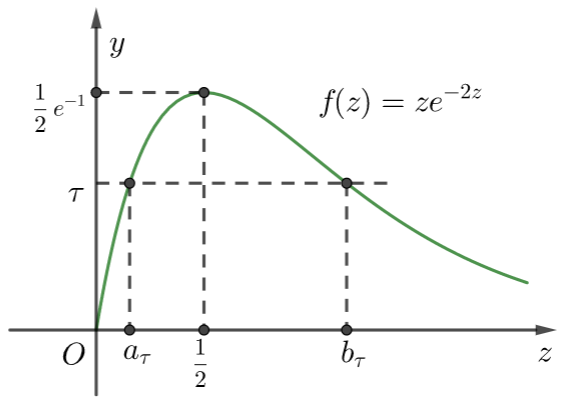}
\caption{The two solutions of the equation $\{z: z {e^{ - 2z}} = \tau\} $.}
\end{figure}

Note that ${{b}^{T}}[\dot{\ell }(
\beta + {b}x) - \ell (\beta )]$ is an increasing function of $x$, then
$S$ cannot be a disjoint union of the two intervals by the restriction condition in \eqref{eq:ie-sbd}, thus $S$ is a
closed interval $x \in [0,\tilde{x}]$. The fact that all $x \in [0,{\|
{\tilde{\beta }} \|_{1}}]$ implies $x \in [0,\tilde{x}]$, therefore we have
\begin{equation*}
Kx \le K\tilde{x} \le {a_{\tau }}.
\end{equation*}

Use (\ref{eq:ie5}) again, it implies $K\tilde x{e^{ - 2K\tilde x}} \le \tau $.  Then,  for $\forall x \in [0,\tilde x] $, by we have
\[{|| {\tilde \beta} ||_1} = || {{b}x} ||_1 = || {b} ||_1 x = x \le \tilde x \le \frac{{{a_\tau }}}{K} = \frac{{{e^{2{a_\tau }}}\tau }}{K} = \frac{{{e^{2{a_\tau }}}(\xi  + 1){d^*}w_{\max }^2}}{{2{w_{\min }}{{\left[ {C(\xi {w_{\max }}/{w_{\min }},H)} \right]}^2}}}\]

It remains to show the oracle inequality of ${D^s}(\hat \beta ,\beta ^*)$. Let $\delta={b}x$ in Lemma~\ref{lem:ie3},  Using the definition of compatibility factor, and observing that ${a_\tau } \ge Kx ~\text{and}~ Kx \ge {|| { \Phi ({{X}})} ||_\infty }|| \delta ||_1$, it derives
\begin{align*}
{e^{ - 2{a_\tau }}}{\left[ {C(\xi {w_{\max }}/{w_{\min }},H)} \right]^2}|| {{{\tilde \beta }_H}} ||_1^2/{d^*} \le {e^{ - 2{a_\tau }}}{{\tilde \beta}^T}\ddot \ell({\beta ^*})\tilde \beta & \le {e^{ - 2Kx}}{{\tilde \beta }^T}\ddot l({\beta ^*})\tilde \beta \\
& \le {e^{ - 2{{|| { \Phi ({X}) } ||}_\infty }|| \delta  ||_1}}{{\tilde \beta }^T}\ddot \ell({\beta ^*})\tilde \beta.
\end{align*}
Apply Lemma~\ref{lem:ie3}, it immediately deduces that
\[{e^{ - 2{a_\tau }}}{\left[ {C(\xi {w_{\max }}/{w_{\min }},H)} \right]^2}|| {{{\tilde \beta }_H}} ||_1^2/{d^*} \le{e^{ - 2{{|| { \Phi ({{X}})} ||}_\infty }|| \delta  ||_1}}{{\tilde \beta }^T}\ddot \ell({\beta ^*})\tilde \beta \le{D^s}(\hat\beta ,\beta ^*) \le \frac{{2\xi {w_{\max }}}}{{\xi  + 1}}{|| {{{\tilde \beta}_H}} ||_1}\]
where the last inequality is from (\ref{eq:ie-sbd}).

Cancel the term ${|| {{{\tilde \beta }_H}} ||_1}$ on both sides of the inequality above, we get
\[{|| {{{\tilde \beta }_H}} ||_1} \le \frac{{2{e^{2{a_\tau }}}\xi {d^*}{w_{\max }}}}{{(\xi  + 1){{\left[ {C(\xi {w_{\max }}/{w_{\min }},H)} \right]}^2}}}.\]
Consequently,
\[{D^s}(\hat \beta ,{\beta ^*}) \le \frac{{2\xi {w_{\max }}}}{{\xi  + 1}}{|| {{{\tilde \beta }_H}} ||_1} \le \frac{{4{e^{2{a_\tau }}}{\xi ^2}{d^*}w_{\max }^2}}{{{{(\xi  + 1)}^2}{{\left[ {C(\xi {w_{\max }}/{w_{\min }},H)} \right]}^2}}}.\]
\end{proof}

Theorem~\ref{theo-oi} indicates that, if we presume ${d^{*}} = O(1)$,
the $l_{1}$-estimation error bound is of the order $\sqrt{\frac{
{\log p}}{n}} $. And this convergence rate is rate-optimal in the
minimax sense; see \cite{Raskutti2011} and \cite{Li2018}. Under some regularity conditions, the weighted Lasso estimates are expected to have the consistency property in the sense
of the $\ell _{1}$-norm when the dimension increases as order
$\exp (o(n))$. Unlike the MLE estimator with convergence rate
$\frac{1}{{\sqrt{n} }}$, in the high-dimensional case we have to
multiply $\sqrt{\log p} $ to the rate of MLE as the price to pay.

\subsubsection{Case of Poisson regression}
In this section, we use the same procedure as in Sect.~\ref{nb} to
prove the case of Poisson regression with weighted Lasso penalty. The
derivation of this type of oracle inequalities is simpler than that in \cite{Ivanoff16}. Similar to Lemma~5 of \cite{Zhang17}, the proof is based on the following key lemma.

\begin{lemma}
\label{eq:Poissonie3}
The Hessian matrix and gradient of the log-likelihood function in the
Poisson regression model are $\dot{\ell }(\beta ) = - \frac{1}{n}
\sum_{i = 1}^{n} {\phi ({{x}_{i}})[} {Y_{i}} - {e^{{\phi ^{T}}(
{{x}_{i}})\beta }}],\ddot{\ell }(\beta ) = \frac{1}{n}\sum_{i =
1}^{n} {\phi ({{x}_{i}}){\phi ^{T}}({{x}_{i}})} {e^{{\phi ^{T}}({{x}
_{i}})\beta }}$. Assume the
identifiability condition: ${\phi ^{T}}({{x}_{i}})(\beta + \delta )={\phi ^{T}}({{x}_{i}})\beta $ implies ${\phi ^{T}}({{x}_{i}})\delta  = 0$, then we have
\begin{equation*}
{D^{s}}(\beta + \delta ,\beta ) \ge {\delta ^{T}}
\ddot{\ell }(\beta ) \delta {e^{ - {{\| {\varPhi ({{X}})} \|}_{\infty }}\| \delta \|_{1}}},
\end{equation*}
where ${\| {\varPhi ({{X}})} \|_{\infty }} = \max \{ |
{\phi _{j}({x_{ij}})}|;1 \le i \le n,1 \le j \le p\} $.
\end{lemma}
\begin{proof}
Without loss of generality, we assume that ${\phi ^{T}}({{x}_{i}})\delta
\ne 0$. By the expression of $\dot{\ell }(\beta )$ and $\ddot{\ell }(
\beta )$, one deduces
\begin{align*}
{\delta ^{T}} \bigl[\dot{\ell }(\beta +\delta ) - \dot{\ell }(\beta
) \bigr] &= {\delta ^{T}} \cdot \frac{1}{n}\sum
_{i = 1}^{n} {\phi ({{x}_{i}}) \bigl(
{e^{{\phi ^{T}}({{x}_{i}})(\beta + \delta )}} - {e^{{\phi ^{T}}({{x}
_{i}})\beta }} \bigr)}
\\
&= {\delta ^{T}} \cdot \frac{1}{n}\sum
_{i = 1}^{n} {\phi ( {{{x}}_{i}}){e^{{\phi ^{T}}({{x}_{i}})\beta }}
\cdot \frac{
{{e^{{\phi ^{T}}({{x}_{i}})\delta }} - 1}}{{{\phi ^{T}}({{x}_{i}})
\delta - 0}}} \cdot {\phi ^{T}}({{x}_{i}})
\delta
\\
& \ge {\delta ^{T}}\frac{1}{n}\sum
_{i = 1}^{n} {\phi ({{x}_{i}}) {\phi
^{T}}({{x}_{i}}){e^{{\phi ^{T}}({{x}_{i}})\beta }}} \delta \cdot
{e^{ - {{\| { {\varPhi ({{X}})}} \|}_{\infty }}\| \delta \|_{1}}},
\end{align*}
where the last inequality is from $
\frac{{{e^{x}} - {e^{y}}}}{{x - y}} \ge {e^{ - |x | \vee |y|}}$.
\end{proof}

\begin{theorem}\label{Poissontheo-oi}
For weighted lasso estimator $\hat \beta$ in Poisson regression, suppose ${|| {\Phi ({{X}})} ||_\infty } \le K$ with some $K > 0$ and let $\tau  = \frac{{K(\xi  + 1){d^*}w_{\max }^2}}{{2{w_{\min }}{{\left[ {C(\xi {w_{\max }}/{w_{\min }},H)} \right]}^2}}} \le {e^{ - 1}}$ with a certain $K > 0$. Then, in the event ${{{|| {\dot \ell({\beta ^{\rm{*}}})} ||}_\infty } \le \frac{{\xi  - 1}}{{\xi  + 1}}{w_{\min }}}$, we have
\[{|| {\hat \beta  - {\beta ^{\rm{*}}}} ||_1} \le \frac{{{e^{{a_\tau}}}(\xi  + 1){d^*}w_{\max }^2}}{{2{w_{\min }}{{\left[ {C(\xi {w_{\max }}/{w_{\min }},H)} \right]}^2}}}\]
where ${a_\tau } \le 1$ is smaller solution of the equation $a {e^{ - a }} = \tau$.
And
\[{D^s}(\hat \beta ,{\beta ^*}) \le \frac{{4{e^{{a_\tau }}}{\xi ^2}{d^*}w_{\max }^2}}{{{{(\xi  + 1)}^2}{{\left[ {C(\xi {w_{\max }}/{w_{\min }},H)} \right]}^2}}}\]
\end{theorem}
\begin{proof}
Note that ${a_\tau }$ is the small solution of $z {e^{ - z}} = \tau $ in Poisson regression. By Lemma~\ref{lem:ulb} and Lemma~\ref{eq:Poissonie3}, the proof is almost exactly the same as proof of Theorem~\ref{theo-oi}.
\end{proof}

\section{Simulations}
\label{s:simu}
\begin{table}[!htb]
\centering
\caption{Means{} of $\ell_1$ error and prediction error}
\label{tab:simu-1}
\begin{tabular}{lcc}
    \hline
    \textbf{Method} & $\Vert \hat{\boldsymbol{\beta}} - \boldsymbol{\beta}^{*} \Vert_{1}$ & \textbf{Prediction error} \\ \hline
    \multicolumn{3}{c}{$p = 20, \rho = 0.5$} \\
    Lasso            &  1.902  &  7.161 \\
    Weight Lasso    &  1.624  &  6.147 \\

    \multicolumn{3}{c}{$p = 20, \rho = 0.8$} \\
    Lasso            &  3.001  &  7.155 \\
    Weight Lasso    &  2.917  &  6.979 \\

    \multicolumn{3}{c}{$p = 100, \rho = 0.5$} \\
    Lasso           &  4.137  &  15.452 \\
    Weight Lasso    &  3.905  &  13.866 \\
    \multicolumn{3}{c}{$p = 100, \rho = 0.8$} \\
     Lasso         &  4.543  &  10.866 \\
    Weight Lasso    &  4.552  &  10.695 \\
    \multicolumn{3}{c}{$p = 200, \rho = 0.5$} \\
     Lasso          &  4.467  &  17.533 \\
    Weight Lasso    &  4.373  &  15.854 \\
    \multicolumn{3}{c}{$p = 200, \rho = 0.8$} \\
   Lasso       &  4.467  &  17.533 \\
     Weight Lasso   &  4.373  &  15.854 \\
    \hline
\end{tabular}
\end{table}

This section aims to compare the performances of the un-weighted Lasso
and weighted Lasso for NB regression on simulated data sets. We use the
\texttt{R} package \texttt{mpath} with the function \texttt{glmregNB}
to fit the un-weighted Lasso estimator of NB regression. The
\texttt{R} package \texttt{lbfgs} is employed to do a weighted
$\ell _{1}$-penalized optimization for NB regression. For weighted Lasso,
we first run an un-weighted Lasso estimator and find the optimal tuning
parameter ${\lambda _{\mathrm{{op}}}}$ by cross-validation. The actual
weights we use are the standardized weights given by
\begin{equation*}
\tilde{w}_{j}: = p \frac{{w_{j}}}{{\sum_{j = 1}^{p} {w_{j}} }}.
\end{equation*}
And then we solve the optimization problem

\begin{equation}
\label{eq:Wlasso1} \hat{\beta }^{(t)}= \argmin_{\beta \in
\mathbb{R}^{p}} \Biggl
\lbrace - \ell (\beta )+ {\lambda _{\mathrm{{op}}}} \sum
_{j=1}^{p} \tilde{w}_{j} |
{{\beta _{j}}} |\Biggr\rbrace .
\end{equation}

Here we first apply the function \texttt{cv.glmregNB()} to do a 10-fold
cross-validation to obtain the optimal penalty parameter $\lambda $. For
comparison, we use the function \texttt{glmregNB} to do
un-weighted Lasso for NB regression and use the estimated $\theta $ as
the initial value for our weighted Lasso algorithms.

For simulations, we simulate 100 random data sets. By optimization with
suitable $\lambda $, we obtain the model with the parameter
$\hat{\beta }_{\lambda _{\mathrm{opt}}}$. Then we estimate the
performance of the $\ell _{1}$ estimation error $\Vert \beta ^{*} -
\hat{\beta }_{\lambda _{\mathrm{opt}}} \Vert _{1}$ and the prediction
error $ \Vert X_{\mathrm{test}} \beta ^{*} - X_{\mathrm{test}}
\hat{\beta }_{\lambda _{\mathrm{opt}}} \Vert _{2}$ on the test data ($X
_{\mathrm{test}}$ of size $n_{\mathrm{test}}$) by the average of the
100-times errors.

For each simulation $n_{\mathrm{train}} = 100$, $n_{\mathrm{test}} = 200$.
We mainly adopt the simulation setting: The predictor variables $X$ are
randomly drawn from ${N}_{p}(0, \varSigma )$, where $\varSigma $ has elements
$\rho _{|j - l|}$ ($j, l = 1, 2, \ldots , p$). The correlation among
predictor variables is controlled by $\rho $ with $\rho = 0.5\mbox{ and }0.8$,
respectively. We assign the true vector as
\begin{equation*}
\beta ^{*} =( 1.25, -0.95, 0.9, -1.1, 0.6, \underbrace{0, \ldots ,
0} _{p - 5} ).
\end{equation*}

The simulation results are displayed in Table~\ref{tab:simu-1}. The
table shows that the proposed weighted Lasso estimators are more
accurate than the un-weighted Lasso estimators. With the help of the
optimal weights, it reflects that the controlling of KKT conditions by
a data-dependent tuning parameter is able to improve the accuracy of the
estimation in aspects of the $\ell _{1}$-estimation, square prediction
errors. It can be also seen that the increasing $p$ will hinder the
$\ell _{1}$-estimation error by the curse of dimensionality.

\section{Appendix: Proofs}

\subsection{Proof of Theorem~\ref{thm:ccdp}}
The proof is divided into two parts. In the first part, we show that
$\nu (A)$ is countably additive and absolutely continuous. Moreover, in the
second part, we will show Eq.~(\ref{eq:xu1}).

\textit{Step} 1.

(1) Note that $P_{k}(A)>0$ and $\sum_{k=0}^{\infty }P_{k}(A)=1$
by assumption 1, we have $0<P_{k}(A)<1$ whenever $0<m(A)<\infty $. And
from assumption 2, we get $P_{0}(A)=1$ if $m(A)=0$. It follows that
$0<P_{0}(A)\leq 1 \text{ if }0<m(A)<\infty $. By assumptions 2 and 4,
one has $\lim_{m(A)\to 0}\nu (A)=\lim_{m(A)\to 0}-
\log {[1-S_{1}(A)]}=0$. Therefore, $\nu (\cdot )$ is absolutely
continuous.

(2) Let $\nu (A):=-\log {P_{0}(A)}$. For $A_{1}\cap A_{2}=\emptyset $,
it implies by assumption 3 that
\begin{equation*}
\nu (A_{1}\cup A_{2})=-\log {P_{0}(A_{1}
\cup A_{2})}=-\log { \bigl[P_{0}(A _{1})
\cdot P_{0}(A_{2}) \bigr]}=\nu (A_{1})+\nu
(A_{2}),
\end{equation*}
which shows that $\nu (\cdot )$ is finitely additive.\vadjust{\goodbreak}

Based on the two arguments in \textit{Step} 1, we have shown that
$\nu ( \cdot )$ is countable additive.

\textit{Step} 2. For convenience, we first consider the case that
$A=I_{\mathbf{a},\mathbf{b}}=(\mathbf{a}, \mathbf{b}]$ is a
$d$-dimensional interval on $\mathbb{R}^{d}$, and then extend it to any
measurable sets. Let $T({x}) =:\nu ({I_{0,{x}}})$, so $T({x})$ is a
$d$ dimensional coordinately monotone transformation with
\begin{equation*}
T(\mathbf{0})=\mathbf{0},\qquad T({x}) \le T \bigl({x}' \bigr)
\quad \text{iff}\quad {x} \le {x}',
\end{equation*}
where the componentwise order symbol (partially ordered symbol) ``$
\le $'' means that each coordinate in ${x}$ is less than or equal to the
corresponding coordinate of ${x}'$ (for example, with $d=2$, we have
$(1,2)\le (3,4)$). This is by the fact that $\nu (\cdot )$ is absolutely
continuous and finite additive, which implies $\nu ({I_{\mathbf{0},
\mathbf{0}}}) =  \lim_{m({I_{\mathbf{0},{x}}}) \to 0}
\nu ({I_{\mathbf{0},{x}}}) = 0$.

The definition of the minimum set of the componentwise ordering relation
lays the base for our study; see Sect.~4.1 of Dinh The Luc \cite{Dinh16}. Let
$A$ be a nonempty set in $\mathbb{R}^{+d}$. A point $\mathbf{a}
\in A$ is called a \emph{Pareto minimal point} of the set $A$ if there
is no point $a' \in A$ such that $\mathbf{a}' \ge \mathbf{a}$ and
$\mathbf{a}' \ne \mathbf{a}$. The sets of Pareto minimal points of
$A$ is denoted by $\operatorname{PMin}(A)$. Next, we define the generalized
inverse ${T^{ - 1}}(t)$ (like some generalizations of the inverse
matrix, it is not unique, we just choose one) by
\begin{equation*}
{x} = ({x_{1}}, \ldots ,{x_{d}}) ={T^{ - 1}}(t)=:
\operatorname{PMin} \bigl\{ \mathbf{a} \in \mathbb{R}^{+d}:t \le T(
\mathbf{a}) \bigr\} .
\end{equation*}

When $d=1$, the generalized inverse ${T^{ - 1}}(t)$ is just a similar
version of the quantile function.

By Theorem~4.1.4 in Dinh The Luc \cite{Dinh16}, if $A$ is a nonempty compact
set, then it has a Pareto minimal point. Set ${x}=T^{-1}(t)$,
${x}+\boldsymbol{\xi }=T^{-1}(t+\tau )$ and define $\varphi _{k}(\tau ,t)=P
_{k}(I_{{x},{x}+\xi })$ for $t=T({x})$, $\tau =T({x}+\boldsymbol{\xi })-T(
{x})$.

Consequently, by part one, finite additivity of $\nu (\cdot )$, it
implies that
\begin{equation*}
\tau =T({x}+\boldsymbol{\xi })-T({x})=\nu (I_{0,{x}+\boldsymbol{\xi }})- \nu
(I_{\mathbf{0},{x}})=\nu (I_{{x},{x}+\boldsymbol{\xi }})=-\log {P_{0}(I
_{{x},{x}+\boldsymbol{\xi }})}.
\end{equation*}
Thus, we have $\varphi _{0}(\tau ,t)=e^{-\tau }$. Let $\varPhi _{k}(\tau ,t)=
\sum_{i=k}^{\infty }\varphi _{i}(\tau ,t)$. Due to assumption 5,
it follows that
\begin{equation}
\label{eq:assumption 5} \alpha _{k}=\lim_{\tau \to 0}
\frac{\varphi _{k}(\tau ,t)}{\varPhi
_{1}(\tau ,t)}=\lim_{\tau \to 0}\frac{\varphi _{k}(\tau ,t)}{
\tau }\quad
\text{for }k\geq 1.
\end{equation}
Applying assumption 3, we get $\varphi _{k}(\tau +h,t)=\sum_{i=0}
^{k}\varphi _{k-i}(\tau ,t)\varphi _{i}(h,t+\tau )$. Subtracting
$\varphi _{k}(\tau ,t)$ and dividing by $h$ ($h > 0$) in the above
equation, then
\begin{equation}
\label{eq:xu3} \frac{\varphi _{k}(\tau +h,t)-\varphi _{k}(\tau ,t)}{h}=\varphi _{k}( \tau ,t)\cdot
\frac{\varphi _{0}(h,t+\tau )-1}{h}+\sum_{i=1} ^{k}
\varphi _{k-i}(\tau ,t)\cdot \frac{\varphi _{i}(h,t+\tau )}{h}.
\end{equation}
Here the meaning of $h$ in $\varphi _{i}(h,t+\tau )$ is that
\begin{equation*}
h=T({x}+\boldsymbol{\xi }+\boldsymbol{\eta })-T({x}+\boldsymbol{\xi })=\nu
(I_{
\mathbf{0},{x}+\boldsymbol{\xi }+\boldsymbol{\eta }})-\nu (I_{\mathbf{0},
{x}+\xi })=\nu (I_{{x}+\xi , {x}+\boldsymbol{\xi }+\boldsymbol{\eta }}),
\end{equation*}
where ${x}+\boldsymbol{\xi }+\boldsymbol{\eta }=T^{-1}(t+\tau +h)$ and
${x}+\boldsymbol{\xi }=T^{-1}(t+\tau )$.\vadjust{\goodbreak}

Let $h \to 0$ in (\ref{eq:xu3}). By using assumption 5 and its
conclusion (\ref{eq:assumption 5}), it implies that
\begin{equation}
\label{eq:xu4} {{\varphi '}_{k}}(\tau ,t) =:
\frac{\partial }{{\partial \tau }}{\varphi _{k}}( \tau ,t) = - {\varphi
_{k}}(\tau ,t) + {\alpha _{1}} {\varphi
_{k - 1}}( \tau ,t) + {\alpha _{2}} {\varphi
_{k - 2}}(\tau ,t) + \cdots + {\alpha _{k}} {\varphi
_{0}}(\tau ,t) ,
\end{equation}
which is a difference--differential equation.

To solve \eqref{eq:xu4}, we need to write it in matrix form,
\[\frac{{\partial {{\bf{P}}_k}(\tau ,t)}}{{\partial \tau }} := \left( {\begin{array}{*{20}{c}}
{{{\varphi '}_k}(\tau ,t)}\\
 \vdots \\
{{{\varphi '}_1}(\tau ,t)}\\
{{{\varphi '}_0}(\tau ,t)}
\end{array}} \right) = \left( {\begin{array}{*{20}{c}}
{ - 1}&{{\alpha _1}}& \cdots &{{\alpha _k}}\\
{}& \ddots & \ddots &\begin{array}{l}
 \vdots \\
{\alpha _1}
\end{array}\\
{}&{}& \ddots &{ - 1}
\end{array}} \right)\left( {\begin{array}{*{20}{c}}
{{\varphi _k}(\tau ,t)}\\
 \vdots \\
{{\varphi _1}(\tau ,t)}\\
{{\varphi _0}(\tau ,t)}
\end{array}} \right) = :{\bf{Q}}{{\bf{P}}_k}(\tau ,t)\]

The general solution is
\begin{equation*}
\mathbf{P}_{k}(\tau , t)=e^{\mathbf{Q}\tau }\cdot \mathbf{c}=: \sum
_{m=0}^{\infty }\frac{\mathbf{Q}^{m}}{m!}\tau
^{m}\cdot \mathbf{c},
\end{equation*}
where $\mathbf{c}$ is a constant vector. To specify $\mathbf{c}$, we
observe that
\begin{equation*}
(0,\ldots ,0,1)^{\mathrm{T}}=\lim_{\tau \to 0}
\mathbf{P}_{k}( \tau , t)=\lim_{\tau \to 0}e^{\mathbf{Q}\tau }
\mathbf{c}= \mathbf{c}.
\end{equation*}
Hence, the $\mathbf{Q}$ can be written as
\begin{equation*}
\mathbf{Q}=-\mathbf{I}_{k+1}+\alpha _{1} \mathbf{N}+
\alpha _{2} \mathbf{N}^{2}+\cdots +\alpha
_{k}\mathbf{N}^{k},
\end{equation*}
where $\mathbf{N}=: \begin{pmatrix}\mathbf{0}_{k\times 1}& \mathbf{I}_k\\0& \mathbf{0}_{1\times k}\end{pmatrix}$. As $i\geq k+1$, we have $\mathbf{N}^{i}=\mathbf{0}$.

We use the expansion of the power of the multinomial
\begin{equation*}
\frac{1}{{m!}}{ \Biggl( - {\mathbf{{I}}_{k + 1}} + \sum
_{i = 1}^{k} {{\alpha _{i}}}
{ \mathbf{{N}}^{i}} \Biggr)^{m}} = \sum
_{{s_{0}} + {s_{1}} + \cdots + {s_{k}} = m} {\frac{{{{( -
{\mathbf{{I}}_{k + 1}})}^{{s_{0}}}}\alpha _{1}^{{s_{1}}} \cdots \alpha
_{k}^{{s_{k}}}}}{{{s_{0}}!{s_{1}}! \cdots {s_{k}}!}}} {\mathbf{{N}}
^{{s_{1}} + 2{s_{2}} + \cdots + k{s_{k}}}}.
\end{equation*}
Then
\begin{align}
{{\mathbf{P}_{k}}(\tau ,t)} &= \sum_{m = 0}^{\infty }{
\sum_{{s_{0}} + {s_{1}} + \cdots + {s_{k}} = m} {\frac{{{{( -
{\mathbf{{I}}_{k + 1}})}^{{s_{0}}}}\alpha _{1}^{{s_{1}}} \cdots \alpha
_{k}^{{s_{k}}}}}{{{s_{0}}!{s_{1}}! \cdots {s_{k}}!}}} {\mathbf{{N}}
^{{s_{1}} + 2{s_{2}} + \cdots + k{s_{k}}}} {\tau ^{{s_{0}} + {s_{1}} +
\cdots + {s_{k}}}} {\mathbf{{c}}}}
\nonumber
\\
&=\sum_{m = 0}^{\infty }{\sum
_{{s_{0}} = 0}^{m} {{\tau ^{
{s_{0}}}}
\frac{{{{( - 1)}^{{s_{0}}}}}}{{{s_{0}}!}}\sum_{{s_{1}} + \cdots + {s_{k}} = m - {s_{0}}} {
\frac{{\alpha _{1}
^{{s_{1}}} \cdots \alpha _{k}^{{s_{k}}}}}{{{s_{1}}! \cdots {s_{k}}!}}} } {\mathbf{{N}}^{{s_{1}} + 2{s_{2}} + \cdots + k{s_{k}}}} {\tau
^{{s_{1}} +
\cdots + {s_{k}}}} {\mathbf{{c}}}}
\nonumber
\\
&= \sum_{{s_{0}} = 0}^{\infty }{\sum
_{m - {s_{0}} = 0} ^{\infty }{{\tau ^{{s_{0}}}}
\frac{{{{( - 1)}^{{s_{0}}}}}}{{{s_{0}}!}} \sum_{{s_{1}} + \cdots + {s_{k}} = m - {s_{0}}} {
\frac{{\alpha
_{1}^{{s_{1}}} \cdots \alpha _{k}^{{s_{k}}}}}{{{s_{1}}!
\cdots {s_{k}}!}}} } {\tau ^{{s_{1}} + \cdots + {s_{k}}}} {\mathbf{{N}}
^{{s_{1}} + 2{s_{2}} + \cdots + k{s_{k}}}} {\mathbf{{c}}}}
\nonumber
\\
(\text{let } r = m - {s_{0}}) & = \sum
_{{s_{0}} = 0}^{
\infty }{{\tau ^{{s_{0}}}}
\frac{{{{( - 1)}^{{s_{0}}}}}}{{{s_{0}}!}} \sum_{r = 0}^{\infty }{
\sum_{{s_{1}} + \cdots + {s_{k}}
= r} {\frac{{\alpha _{1}^{{s_{1}}} \cdots \alpha _{k}^{{s_{k}}}}}{{{s
_{1}}! \cdots {s_{k}}!}}} } {\tau
^{{s_{1}} + \cdots + {s_{k}}}} {\mathbf{{N}} ^{{s_{1}} + 2{s_{2}} + \cdots + k{s_{k}}}} {\mathbf{{c}}}}
\nonumber
\\
&= {e^{ - \tau }}\sum_{r = 0}^{\infty }{
\sum_{{s_{1}} + \cdots + {s_{k}} = r} {\frac{{\alpha _{1}^{{s_{1}}}
\cdots \alpha _{k}^{{s_{k}}}}}{{{s_{1}}! \cdots {s_{k}}!}}} } {\tau
^{{s
_{1}} + \cdots + {s_{k}}}} {\mathbf{{N}}^{{s_{1}} + 2{s_{2}} + \cdots
+ k{s_{k}}}} {\mathbf{{c}}}
\nonumber
\\
&=\sum_{l = 0}^{\infty }{\sum
_{R(s,k) = l} {\frac{
{\alpha _{1}^{{s_{1}}} \cdots \alpha _{k}^{{s_{k}}}}}{{{s_{1}}! \cdots
{s_{k}}!}}{\tau ^{{s_{1}} + \cdots + {s_{k}}}}
{e^{ - \tau }} {\mathbf{{N}} ^{l}} {\mathbf{{c}}}} }
\nonumber
\\
&= \sum_{l = 0}^{k} {\sum
_{R(s,k)
= l} {\frac{{\alpha _{1}^{{s_{1}}} \cdots \alpha _{k}^{{s_{k}}}}}{{{s
_{1}}! \cdots {s_{k}}!}}{\tau ^{{s_{1}} + \cdots + {s_{k}}}}
{e^{ -
\tau }} {\mathbf{{N}}^{l}} {\mathbf{{c}}}}
},\label{eq:Pk}
\end{align}
where $R(s,k)=:\sum_{t=1}^{k} ts_{t}$ and the last equality is
obtained by the fact that ${\mathbf{N}^{l}}\mathbf{c} = (0,
\ldots ,0,1,\allowbreak \underbrace{0, \ldots ,0}_{l})^{T}$ and $\mathbf{N}^{l}=
\mathbf{0}$ as $l\geq k+1$.

Choose the first element in \eqref{eq:Pk}. We show that $\varphi _{k}(
\tau , t)=\sum_{R(s,k)=k}\frac{\alpha _{1}^{s_{1}}\cdots \alpha
_{k}^{s_{k}}}{s_{1}!\cdots s_{k}!}\tau ^{s_{1}+\cdots +s_{k}} e^{-
\tau }$. By the relation of $\varphi _{k}(\tau , t)$ and $P_{k}(A)$, we
obtain
\begin{equation}
\label{eq:xu2} P_{k}(I_{0,\boldsymbol{\xi }})=\sum
_{R(s,k)=k}\frac{\alpha _{1}^{s
_{1}}\cdots \alpha _{k}^{s_{k}}}{s_{1}!\cdots s_{k}!} \bigl[\nu (I_{0,
\boldsymbol{\xi }})
\bigr]^{s_{1}+\cdots +s_{k}}\cdot e^{-\nu (I_{0,
\boldsymbol{\xi }})}.
\end{equation}

\textit{Step} 3. We need to show the countable additivity in terms of
(\ref{eq:xu2}). Let $A_{1}$, $A_{2}$ be two disjoint intervals, we first
show that
\begin{equation}
\label{eq:xu8} P_{k}(A_{1}\cup A_{2})=
\sum_{R(s,k)=k}\frac{\alpha _{1}^{s_{1}}
\cdots \alpha _{k}^{s_{k}}}{s_{1}!\cdots s_{k}!} \bigl[\nu
(A_{1}\cup A_{2}) \bigr]^{s
_{1}+\cdots +s_{k}}
e^{-\nu (A_{1}\cup A_{2})} ,
\end{equation}
which means that (\ref{eq:xu2}) is closed under a finite union of the
$A_{i}$.

By assumption 3, to show (\ref{eq:xu8}), it is sufficient to verify
$\sum_{i=0}^{k}P_{k-i}(A_{1})P_{i}(A_{2})$ equals the right-hand
side of (\ref{eq:xu8}). From the m.g.f. of the DCP distribution with
complex $\theta $, we have
\begin{equation*}
{M_{{A_{i}}}}(\theta ) =: \mathrm{{E}} {e^{ - \theta N({A_{i}})}} = \exp
\Biggl\{ \lambda ({A_{i}})\sum_{a = 1}^{\infty }{{
\alpha _{a}} \bigl( {e^{ - a\theta }} - 1 \bigr)} \Biggr\}\quad
\text{for } i=1,2.
\end{equation*}
And by the definition of m.g.f. and assumption 3,
\begin{align*}
{M_{{A_{1}} \cup {A_{2}}}}(\theta ) &= \sum_{k = 0}^{ + \infty}
{{P_{k}}({A_{1}} \cup {A_{2}}){e^{ - k\theta }}}
= \sum_{k =
0}^{ + \infty } { \Biggl( {\sum
_{i = 0}^{k} {{P_{k - i}}} (
{A_{1}}){P_{i}}({A_{2}})} \Biggr)}
{e^{ - k\theta }}
\\
&= \Biggl( {\sum_{j = 0}^{ + \infty }
{{P_{j}}({A_{1}}){e^{ -
j\theta }}} } \Biggr)
\Biggl( {\sum_{s = 0}^{ + \infty } {{P
_{s}}({A_{2}}){e^{ - s\theta }}} } \Biggr)
\\
&={M_{{A_{1}}}}(\theta ){M_{{A_{2}}}}(\theta )=\exp \Biggl\{
\bigl[\nu ({A_{1}}) + \nu ({A_{1}}) \bigr]\sum
_{a = 1}^{\infty }{{\alpha _{a}}
\bigl({e^{ - a
\theta }} - 1 \bigr)} \Biggr\}
\\
&=\exp \Biggl\{ \bigl[\nu ({A_{1}} + {A_{2}}) \bigr]
\sum_{a = 1}^{\infty } {{\alpha
_{a}} \bigl({e^{ - a\theta }} - 1 \bigr)} \Biggr\}
\\
&=\sum_{k = 0}^{\infty }{\sum
_{R(s,k) = k} {\frac{
{\alpha _{1}^{{s_{1}}} \cdots \alpha _{k}^{{s_{k}}}}}{{{s_{1}}! \cdots
{s_{k}}!}}} {{ \bigl[\nu
({A_{1}} \cup {A_{2}}) \bigr]}^{{s_{1}} + \cdots + {s
_{k}}}}
{e^{ -\nu ({A_{1}} \cup {A_{2}})}}} {\mathrm{{e}}^{ - k\theta }}.
\end{align*}
Thus, we show (\ref{eq:xu8}).

Finally, let $E_{n}=\bigcup_{i=1}^{n} A_{i}$,
$E=\bigcup_{i=1}^{\infty }A_{i}$ for disjoint intervals $A_{1}, A_{2}, \ldots $. We use the following lemma in Copeland and Regan \cite{Copeland36}.

\begin{lemma}
If we have assumptions 1--4, let ${e_{n}} = \{ E\backslash (E \cap
{E_{n}})\} \cup \{ {E_{n}}\backslash (E \cap {E_{n}})\} $. Then\vspace*{-1.5pt}
\begin{equation*}
\lim_{n\to \infty } m(e_{n})=0 \quad \textit{implies}
\quad \lim_{n\to \infty } P_{k}(E_{n})=P_{k}(E).
\end{equation*}
\end{lemma}

To see this, since $e_{n}=\bigcup_{i=n+1}^{\infty }A_{i}$, we
have $m(e_{n})=\sum_{i=n+1}^{\infty }m(A_{i}) \to 0$ as
$n \to \infty $. Thus, we obtain the countable additivity with respect
to (\ref{eq:xu2}).

\subsection{Proof of Corollary~\ref{col:ce}}
Let $C{P_{i,{r_{i}}}}(A) =: \sum_{k = 1}^{r_{i}}
{k{N_{i,k}}(A)}$, where $\{{N_{i,k}}(A)\}_{i=1}^{n}$ are independent Poisson point process with intensity ${{\alpha _{k}(i)}\int _{A} {\lambda (x)} \,dx}$ for
each fixed $k$. By Campbell's theorem, it follows that
\begin{align}\label{eq:LF}
{\rm{E}}\exp \{  - \sum\limits_{i = 1}^n {\int_{{S_i}} {f_i(x)} C{P_{i,{r_i}}}(dx)} \} & = {\rm{E}}\exp \{  - \sum\limits_{i = 1}^n {\int_{{S_i}} {f_i(x)} \sum\limits_{k = 1}^{{r_i}} {k{N_{i,k}}(dx)} } \}\nonumber\\
& = \prod\limits_{i = 1}^n {\prod\limits_{k = 1}^{{r_i}} {\rm{E}} \exp \{  - \int_{{S_i}} {kf_i(x)} {N_{i,k}}(dx)\} }\nonumber\\
& = \prod\limits_{i = 1}^n {\prod\limits_{k = 1}^{{r_i}} {\exp \{ \int_{{S_i}} {[{e^{ - kf_i(x)}}}  - 1]{\alpha _k}(i)\lambda (x)dx\} } }\nonumber\\
& = \exp \{ \sum\limits_{i = 1}^n {\sum\limits_{k = 1}^{{r_i}} {\int_{{S_i}} {[{e^{ - kf_i(x)}}}  - 1]{\alpha _k}(i)\lambda (x)dx} } \}.
\end{align}
For $\eta>0$, define the normalized exponential transform of the stochastic integral by $D_{\{ r_i\} }(\eta )$:
\begin{equation} \label{eq:ie1A}
\log{D_{\{ r_i\} }(\eta )}: = \eta \sum\limits_{i = 1}^n {\int_{{S_i}} {f_i(x)} [C{P_{i,{r_i}}}(dx) - \sum\limits_{k = 1}^{{r_i}} {k{\alpha _k}(i)\lambda (x)} ]dx}  - \int_{{S_i}} {\sum\limits_{k = 1}^{{r_i}} {[{e^{kf_i(x)}} - k\eta f_i(x) - 1]{\alpha _k}(i)\lambda (x)} } dx.
\end{equation}
Therefore, we obtain ${\rm{E}} D_{\{ r\} }(\eta )= 1$ by \eqref{eq:LF}.

It follows from (\ref{eq:ie1A}) and Markov's inequality, we have
\begin{align} \label{eq:ie2A}
&P\left( {\eta \sum\limits_{i = 1}^n {\int_{{S_i}} {f_i(x)} [C{P_{i,{r_i}}}(dx) - \sum\limits_{k = 1}^{{r_i}} {k{\alpha _k}(i)\lambda (x)} ]dx}  \ge \sum\limits_{i = 1}^n {\int_{{S_i}} {\sum\limits_{k = 1}^{{r_i}} {[{e^{kf_i(x)}} - k\eta f_i(x) - 1]{\alpha _k}(i)\lambda (x)} } dx}  + ny} \right)\nonumber\\
& = P(D_{\{ r\} }(\eta ) \ge {e^{ ny}}) \le {e^{ - ny}}.
 \end{align}
By inequality \eqref{eq:111}, the (\ref{eq:ie2}) implies
\begin{equation} \label{eq:222A}
P\left( {\sum\limits_{i = 1}^n {\int_{{S_i}} {f_i(x)} [C{P_{i,{r_i}}}(dx) - \sum\limits_{k = 1}^{{r_i}} {k{\alpha _k}(i)\lambda (x)} ]dx}  \ge \sum\limits_{i = 1}^n {\int_{S_i} {\sum\limits_{k = 1}^{{r_i}} {{\alpha _k}(i)\left( {\frac{1}{2}\frac{{{k^2}\eta {f_i^2}(x)\lambda (x)}}{{1 - \frac{1}{3}k\eta {{\left\| f_i \right\|}_\infty }}}dx} + \frac{{y}}{\eta } \right)} } } } \right) \le {e^{ - ny}}.
\end{equation}
Let ${V_{i,f_i}} = \int_{S_i} {{f_i^2}(x)\lambda (x)dx} $ for $i=1,2,\cdots,n$. By basic inequality, we have
\begin{align} \label{eq:333A}
~~~~&\sum\limits_{i = 1}^n {\sum\limits_{k = 1}^{{r_i}} {{\alpha _k}(i)\left( {\frac{1}{2}\frac{{\eta {k^2}{V_{i,f_i}}}}{{1 - \frac{1}{3}k\eta {{\left\| f_i \right\|}_\infty }}} + \frac{y}{\eta }} \right)} }\nonumber\\
 & = \sum\limits_{i = 1}^n {\sum\limits_{k = 1}^{{r_i}} {{\alpha _k}(i)\left( {\frac{1}{2}\frac{{\eta {k^2}{V_{i,f_i}}}}{{1 - \frac{1}{3}k\eta {{\left\| f_i \right\|}_\infty }}} + \frac{y}{\eta }(1 - \frac{1}{3}k\eta {{\left\| f_i \right\|}_\infty }) + \frac{{ky}}{3}{{\left\| f_i \right\|}_\infty }} \right)} } \nonumber\\
& \ge \sum\limits_{i = 1}^n {\sum\limits_{k = 1}^{{r_i}} {{\alpha _k}(i)\left( {k\sqrt {2{V_{i,f_i}}y}  + \frac{{ky}}{3}{{\left\| f_i \right\|}_\infty }} \right)}  = \sum\limits_{i = 1}^n\sum\limits_{k = 1}^{{r_i}} {k{\alpha _k}} (i)[\sqrt {2{V_{i,f}}y}  + \frac{y}{3}{{\left\| f_i \right\|}_\infty }]}.
\end{align}
Optimizing $\eta$ in \eqref{eq:222A} by the \eqref{eq:333A}, we have
\begin{equation} \label{eq:ie22A}
P\left( {\sum\limits_{i = 1}^n {\int_{{S_i}} {f(x)} [C{P_{i,{r_i}}}(dx) - \sum\limits_{k = 1}^{{r_i}} {k{\alpha _k}(i)\lambda (x)} ]dx}  \ge \sum\limits_{i = 1}^n {\sum\limits_{k = 1}^{{r_i}} {k{\alpha _k}} (i)(\sqrt {2y{V_{i,f_i}}} }  + \frac{y}{3}{{\left\| f_i \right\|}_\infty })} \right) \le {e^{ - ny}}.
\end{equation}
For $i=1,2,\cdots$, let $r_i \to \infty $ in (\ref{eq:ie22A}), then $CP_{i,r_{i}}(A) \xrightarrow{d} CP_{i}(A)$. It implies the concentration inequality (\ref{eq:ie0A}).

Define ${c_{1n}} := \sum\limits_{i = 1}^n {{\mu _i}} \sqrt {2{V_{i,f}}} ,{c_{2n}} := \sum\limits_{i = 1}^n {\frac{{{\mu _i}}}{3}}{\left\| f_i \right\|_\infty }$, let $t={{c_{1n}}\sqrt {y}  + {c_{2n}}y}$, and we obtain the expression of $y$ by solving a quadratic equation.

\section{Summary and discussion}
The paper contributes three parts. In the first part, we provide a characterization
of discrete compound Poisson point processes (DCPP) in the same fashion as it has been
done in 1936 by Arthur H. Copeland and Francis Regan for the Poisson process. The
second part is focused on deriving concentration inequalities for DCPP, which
are applied in the third part to derive optimal oracle inequalities of weighted Lasso in high-dimensional NB regression.

Oracle inequalities for discrete distributions are statistically useful in both the non-asymptotic and asymptotical analysis of count data regression. The motivation for deriving new concentration inequalities for DCP processes is that we want to show that the KKT conditions of the NB regressions with Lasso penalty (with a tuning parameter $\lambda$) is a high-probability event. The KKT conditions is based on the concentration equality for centralized weighted sum of NB random variables. However, existing concentration results are not friendly applicable to the aim of optimal inference procedure. The optimal inference procedure here is to choose a optimal tuning parameter for Lasso estimates in high-dimensional NB regressions, which lead to the minimax optimal convergence for the estimator by considering the oracle inequalities for $\ell_1$-estimation error. In the future, it is of interest to study the concentration inequalities for other extended Poisson distributions in regression analysis such as Conway-Maxwell-Poisson distribution, see \cite{Li2018} for the existing probability properties. In another direction, our proposed concentration inequalities would be desirable to extend to establish oracle inequalities for penalized projection estimators studied by \cite{Reynaud-Bouret03}, when considering the intensity of some inhomogeneous compound Poisson processes.

\section{Acknowledgments}
This work is part of the doctoral thesis of the first author who would like to show sincere gratefulness to one of my advisor Professor Jinzhu Jia for his guidance, and candidate Dr. Xiangyang Li for his assistance of simulations. The authors also express their thanks to the two anonymous referees for their valuable comments which greatly improved the quality of our manuscript. In writing the early version of this paper, Professor Patricia Reynaud-Bouret provided several helpful comments and materials about the concentration inequalities, to whom warm thanks are due.

%
%
%
%
%
%

\end{document}